\documentclass{amsart}
\usepackage{graphicx}

\newtheorem{theorem}{Theorem}[section]
\newtheorem{proposition}{Proposition}[section]
\newtheorem{lemma}[theorem]{Lemma}
\newtheorem{corollary}[theorem]{Corollary}

\theoremstyle{definition}
\newtheorem{definition}[theorem]{Definition}

\theoremstyle{remark}
\newtheorem{remark}[theorem]{Remark}

\numberwithin{equation}{section}

\newcommand{\abs}[1]{\lvert#1\rvert}

\usepackage{amsfonts,amstext,amssymb,amsthm,mathrsfs,comment,newlfont,enumerate,graphicx, graphics,color,url,mathtools}
\usepackage[usenames,dvipsnames]{xcolor}
\definecolor{BlueFonse}{rgb}{0,0,1}
\definecolor{BlueFonse1}{cmyk}{1,0,0,0.7}
\usepackage[pagebackref, colorlinks=true,linkcolor=Blue,citecolor=ForestGreen]{hyperref}
\usepackage{hyperref}
\usepackage{cleveref}
\hypersetup{
    colorlinks=true,
    linkcolor=blue,
    filecolor=magenta,      
    urlcolor=cyan,
    pdftitle={Sharelatex Example},
    bookmarks=true,
    pdfpagemode=FullScreen,
    }
\usepackage{url}

\usepackage{tikz}
\usetikzlibrary{calc}
\usetikzlibrary{matrix,arrows,decorations.pathmorphing,decorations.pathreplacing,shapes.misc}
\usetikzlibrary{decorations.markings}
\usetikzlibrary{patterns}
\usetikzlibrary{snakes}

\usepackage{these}
\usepackage[normalem]{ulem}

\begin{document}

\title[Poincar\'e-Sobolev type inequalities on intrinsic varifolds]{Poincar\'e and Sobolev type inequalities for intrinsic rectifiable varifolds}

\author{Julio Cesar Correa Hoyos}
\address{Instituto de Matem\'atica e estat\'istica, Universidade de S\~ao Paulo, S\~ao Paulo}
\email{jccorrea@ime.usp.br}

\date{January 24, 2020}


\maketitle
\begin{abstract}
We prove a Poincar\'e, and a general Sobolev type inequalities for functions with compact support defined on a $k$-rectifiable varifold $V$ defined on a complete Riemannian manifold with positive injectivity radius and sectional curvature bounded above. Our techniques allow us to consider Riemannian manifolds $(M^n,g)$ with $g$ of class $C^2$ or more regular, avoiding the use of Nash's isometric embedding theorem. Our analysis permits to do some quite important fragments of geometric measure theory also for those Riemannian manifolds carrying a $C^2$ metric $g$, that is not $C^{k+\alpha}$ with $k+\alpha>2$. The class of varifolds we consider are those which first variation $\delta V$ lies in an appropriate Lebesgue space $L^p$ with respect to its weight measure $\|V\|$ with the exponent $p\in\R$ satisfying $p>k$.         
\end{abstract}


\maketitle

\section{introduction} 

The ordinary Sobolev inequality has been studied for many years in the Euclidean case, as well as in the Riemannian case, and its prominent role in the theory of partial differential equations is well known. In \cite{miranda1967diseguaglianze} Miranda obtained a Sobolev inequality for functions defined on minimal graphs, The Mirandas proof follows from the Isoperimetric inequality proved in \cite{FedererFleming} for integer currents and a procedure introduced by De Giorgi in \cite{BombieriMirandaDegiorgi}. Bombieri used a refined version of this new inequality, De Giorgi and Miranda (employing the isoperimetric inequality of \cite{FedererFleming}), to derive gradient bounds for solutions to the minimal surface equation (see \cite{bombieri1969maggiorazione}).\\

 In \cite{SimonMichael}, Michael and Simon prove a general Sobolev type inequality, which proof follows an argument which is, in some aspects, evocative of the potential theory. That inequality is obtained on what might be termed a generalized manifold and in particular, the classical Sobolev inequality, a Sobolev inequality on graphs of weak solutions to the mean curvature equation, and a Sobolev inequality on arbitrary $C^2$ submanifolds of $\mathbb{R}^n$ (of arbitrary co-dimension) are derived.\\

On the other hand, Allard in a pioneering work \cite{Allard} proved a Sobolev type inequality for non-negative, compact supported functions defined on a varifold $V$, whose first variation $\delta V$ lies in an appropriate Lebesgue space endowed with the measure $\|\delta V\|$, by generalization of the Isoperimetric inequality for integer currents in \cite{FedererFleming}. The proof of this inequality given by Allards follows from the monotonicity formula derived from the computation of the first variation of a varifold for a suitable perturbation of a radial vector field.\\

Following the ideas of Michael and Simon, Hoffman and Spruk proved in \cite{Hoffman1974sobolev} a general Sobolev inequality for submanifolds $N$ of a Riemannian manifold $M$, satisfying geometric restrictions involving the volume of $M$, the sectional curvatures and the injectivity radius of $M$. Their proof is inspired in the Michael and Simon work, therefore, is an extrinsic perspective.\\

Since this general Sobolev inequality has been largely studied in different contexts from an extrinsic point of view, a natural question is whether or not this kind of inequality remains valid from an intrinsic point o view. 
For functions with compact support on a varifold $V$ whose first variation $\delta V$ lies in an appropriate Lebesgue space with respect to $\|\delta V\|$.\\

In this paper, we show an intrinsic Riemannian analog to the Allard result, considering a $k$-dimensional varifold $V$ defined in an $n$-dimensional Riemannian manifold $(M^n,g)$ (with $1\le k\le n$) defined intrinsically. We achieve this goal by recovering a monotonicity inequality (instead of monotonicity equality) in this intrinsic Riemannian context, which takes into account the bounds on the geometry of $M$. Then we follow the ideas of Simon and Michael in \cite{SimonMichael} and \cite{Simon}, to get a local version of the desired inequality. Finally, The Sobolev type inequality is then obtained by a standard covering argument.

\subsection{A quick overview} 
We present a natural extension of Allard's work. In fact, instead of define a general varifold on a Riemannian manifold via an isometric embedding (i.e., as a Radon measure on $G_k(i(M))$ where $i:M^n\to U\subset\R^N$ is an isometric embedding), is defined as a nonnegative, real extended valued, Radon measure on $G_k(M^n)$, the Grassmannian manifold whose underlying set is the union of the sets of $k$-dimensional subspaces of $T_xM$ as $x$ varies on $M^n$. This point of view has a consequence: more freedom on the regularity of the metric (i.e., our theory holds even for the case of only $C^2$ metrics, when the Nash embedding does not exists). In fact, there is a gap in the theory of isometric embeddings in $\R^N$, precisely in the case when the metric is $C^2$ but not $C^{2+\alpha}$ with $\alpha>0$ it is not known whether an isometric immersion into Euclidean space exists. On the other hand, whenever the metric is $C^{k,\alpha}$ with $k+\alpha>2$ there exists isometric embeddings of class $C^{k+\alpha}$. If the metric is $C^{k, \alpha}$ with $k+\alpha<2$ then there are isometric embeddings $C^{1+\frac{k+\alpha}2}$. The first theorem is proved with the aid of the "hard implicit function theorem", à la Nash-Moser. The second is proved using the technics of the first paper by John Nash about isometric embeddings, compare \cite{NashIsometricEmbedding1954}. This freedom constitutes, as shown in \cite{NardulliBBMS} and \cite{NardulliOsorioIMRN}, a powerful tool to tackle problems in noncompact Riemannian manifolds of bounded geometry, or compact Riemannian manifolds with variable metric and also, possibly noncompact Riemannian manifolds of bounded geometry and variable metrics.\\

Our main goal is to reproduce Allard's Theorem $7.3$ in this new context along the lines of  the proofs of Theorems $18.5$ and $18.6$ of \cite{Simon}. To do so, as in \cite{Hoffman1974sobolev}, our ambient manifold must satisfy some geometric conditions, namely positive injectivity radius and sectional curvature bounded above, which from now on we refer as \emph{bounded geometry}, see Definition \ref{Def:GeometricConditions}. 
On the other hand, since we are interested in varifolds satisfying the conditions of the celebrated Theorem $8.1$ \cite{Allard} (from now on called \textit{Allard Conditions} \hyperlink{$(AC)$}{$(AC)$}, see Definition \ref{defallardcondition}) is necessary to study the properties of varifolds satisfying these conditions. Roughly speaking, we are interested in varifolds whose first variation has no boundary term, having generalized mean curvature belonging to $L^p$ with respect to $\|\delta V\|$ for some $p>k$, and whose density ratio is close to that of an Euclidean disk in a small ball, and density bounded below far from $0$, $\|V\|$-a.e.  

The main contributions of this paper are:
\begin{enumerate}[$(i)$]
\item to define rectifiable sets, exploiting the local structure of the ambient manifold, and so, define rectifiable varifolds ans in Chapter $4$ Definition... of \cite{Simon}.
\item to give an intrinsic $L^p$-monotonicity formula valid for manifolds with bounded geometry and not only in $\R^N$, without using Nash's isometric embedding Theorem. See Theorem \ref{FWMI}.\item to prove, intrinsically, Poincar\'e and Sobolev type inequalities for $C^1(S)$ non-negative functions defined over a $\Ha^k_g$-rectifiable set $S\subset M$. Compare Theorems \ref{poincare} and \ref{sobolev}.
\item to prove that the best constant $C>0$ in the extrinsic Sobolev inequality \eqref{Eq:IntrinsicSobolevInequality} depends only on the dimension $k$ of $V$, which is a remarkable fact, since one a priori expects that $C$ could depend also on the bounds of the geometry of the ambient manifold $n,k,inj_M,b$. For the meaning of the former constants see Definition \ref{Def:GeometricConditions}. 
\end{enumerate}
The rest of this introduction will describe in more detail the contributions of this paper.
\subsubsection{Monotonicity Inequality}
Given a general varifold $V$, the classic way to deduce a "monotonicity formula" is to compute the first variation of $V$ at a perturbations, by smooth functions, of (Euclidean) radial vector fields, then seems natural to consider vector fields of the form $\gamma(\frac u s)\lp u\nabla u\rp(x)$, where $\gamma(t)$  is a smooth real valued function which vanishes for large values of the variable $t$, and $u(x):=dist_{(M,g)}(x,\xi)$ for some $\xi\in M$.\\
The classic calculation uses the fact that the divergence of  a radial vector field over a $k$ dimensional plane is exactly $k$, which is not true in general Riemannian manifolds, so it is here where the bounded geometry plays a central role, since the Rauch comparison Theorem gives bounds on such quantity. The price to pay in making our monotonicity formula using this comparison geometry argument is that there is no equality anymore, instead we have the following inequality. 
\begin{Res}\label{FWMI}
Let $(M^n,g)$ be a complete Riemannian manifold with $g$ of class at least $C^2$ having bounded geometry  and Levi-Civita connection $\nabla_g$, such that $r_0\cot_b(r_0)>0$, and  let $V\in\V_k(M^n)$ satisfying \hyperlink{$(AC)$}{$(AC)$}, then for any $0<s<r_0$, and $h\in C^1(M)$ non-negative. There exists a constant $c=c(s,b)\in]0,1[$ such that, if we set $u(x)=r_{\xi}(x)=dist_{(M,g)}(x,\xi)$ we have for all $0<s<r_0$

\begin{equation}\label{eqch1.5}
\begin{split}
\frac{d}{ds}\lp\frac{1}{s^k}\int_{B_g(\xi,s)}h(y) d\|V\|(y)\rp & \geq \frac{d}{ds}\int_{G_k(B_g(\xi,s))}h(y)\frac{\left|\nabla^{S^{\bot}}u\right|^2_g}{r_{\xi}^k}dV(y,T)\\
& + \frac{1}{s^{k+1}}\lp\int_{G_k(B_g(\xi,s))}\Ll\nabla h(y),(u\nabla u)(y)\Rl_gdV(y,T)\rp\\
& + \frac{(c(s)-1)}{s}\frac{k}{s^k}\int_{B_g(\xi,s)}h(y)d\|V\|(y)\\
& +  \frac{1}{s^{k+1}}\int_{B_g(\xi,s)}\Ll H_g,h(y)(u\nabla u)(y)\Rl_gd\|V\|(y).
\end{split}
\end{equation}

Here $H_g$ is as in Proposition \ref{firstvariationrepresentation}.
\end{Res}
\subsubsection{Rectifiable varifolds}
In this subsection we define the main object of our investigation, namely rectifiable varifolds. We follow here the approach of Chapter $4$ of \cite{Simon} without treating the rectifiability of general varifolds under conditions of the first variation and the density of their weights. Before to give the following definition it is worth to recall here the classical Rademacher Theorem.
\begin{theorem}[Rademacher]
 If \(f\) is Lipschitz on \(\mathbb{R}^{n},\) then \(f\) is differentiable \(\mathcal{L}^{n}\) -almost everywhere; that is, the gradient \(\nabla f(x)\)
 exists and
\[
\lim _{y \rightarrow x} \frac{f(y)-f(x)-\nabla f(x) \cdot(y-x)}{|y-x|}=0,
\]
for  $\mathcal{L}^{n}$-a . e .$x \in \mathbb{R}^{n}$.
\end{theorem}
\begin{definition}[cf. \cite{AmbrosioK}, Definition 5.3, pg. 536]\label{ambrosiorectifiable}
Let $(M^n,g)$ be a complete Riemannian manifold. We say that a Borel set $S\subset M$ is \emph{countable $\Ha^k_g$-rectifiable} if there exists countable many Lipschitz functions $f_j:\R^k\to M^n$ such that
\[
\Ha^k_g\lp S\setminus\bigcup_{j\in\N}f_j(\R^k)\rp=0.
\]
For any $x\in\bigcup_{j\in\N}f_j(\R^k)$ such that $f_j$ is differentiable at $y=f_j^{-1}(x)$, we define the approximate tangent space $\Tan{k}{S}{x}$ as $df_y(\R^k)\le T_yM^n$.
\end{definition}
\begin{remark} Observe that the preceding definition is well posed, since, if $x\in f_j(\R^k)\cap f_i(\R^k)$ and $y_l\in f_l^{-1}(x)$ for $l\in\{i,j\}$, then  $df_{y_i}(\R^k)=df_{y_j}(\R^k)$. The reader could consult \cite{AmbrosioK}, Definition 5.5, pg. 536.
\end{remark}

\begin{definition}[c.f. \cite{Simon} Chapter $4$, pg. 77] Let $S$ be a countably $\Ha^k_g$ -rectifiable,  subset of $\lp M^n,g\rp$ and let $\theta\in L_{loc}^1(S,\Ha^k_g)$. Corresponding to such a pair $(S, \theta)$ we define the rectifiable $k$-varifold $\rv{S}{\theta}$ to be,  the equivalence class of all
pairs $(\tilde{S}, \tilde{\theta}),$ where $\tilde{S}$ is countably $\Ha^k_g$-rectifiable with $\mathcal{H}^{k}_g((S \Delta \widetilde{S}) )=0$ and where $\tilde{\theta}=\theta$, $\mathcal{H}^{k}_g-$ a.e. on $S \cap \tilde{S}$.
\end{definition}

So, we can naturally induce a (general) varifold from a rectifiable one as follows.
\begin{definition}
Given an $\Ha^k_g$-rectifiable varifold $\rv{S}{\theta}$ on $M^n$ there is a corresponding (general) $k$-varifold $V\in\V_k(M^n)$ (also denoted by $\rv{S}{\theta}$),  defined by
$$
V(A):=\rv{S}{\theta}(A)=\Ha_g^k\llcorner\theta(\pi(\left\{\left(x, T_{x} M\right): x \in S_{*}\right\}\cap A)), \quad A \subset G_{k}(M^n),
$$
where $S_{*}$ is the set of $x \in S$ such that $S$ has an approximate tangent space $T_{x} S$ with respect to $\theta$ at $x$. Evidently $\rv{S}{\theta}$, so defined, has weight measure $\|\rv{S}{\theta}\|=\mathcal{H}_g^{k} \llcorner\theta$.
\end{definition}
\subsubsection{Poincar\'e and Sobolev inequalities}
In the special case in which, $V=\rv{S}{\theta}$ is a rectifiable varifold such that the mean vector field belongs to a certain Lebesgue space $L^p$ with $p>k$, the Theorem \eqref{FWMI} implies that, for $h\in C^1(S)$, $h\ge0$,
\[
h(\xi)\le e^{\lp\Lambda+c^*k\rp\rho}\lp\frac1{\omega_k}\rho^k\int_{\B{\xi}{\rho}}hd\|V\|+\int_{B_g(\xi,\rho)}\frac{\lmo\nabla^S_g h\rmo}{r^{k-1}_{\xi}}d\|V\|\rp,
\]
for all $\xi\in\spt\|V\|$ and for all $0<\rho<r_0$. From this, together with an approximation argument and Fubini's Theorem we deduce the following Poincar\'e inequality.
\begin{Res}\label{poincare}
Let $(M^n,g)$ be a complete Riemannian manifold with bounded geometry, let $V:=\rv{S}{\theta}$ a rectifiable varifold satisfying \hyperlink{$(AC)$}{$(AC)$}. Suppose: $h\in C^1(M)$, $h\ge 0$, $B_g(\xi,2\rho)\subset B_g(\xi, r_0)$ for $\xi\in S$ fixed, $\lmo H_g\rmo_g\le\Lambda$ for some $\Lambda>0$, $\theta>1$ $\|V\|$-a.e. in $B_g(\xi,r_0)$ and for some $0<\alpha<1$
\begin{equation}
\|V\|\lp\left\{ x\in B_g\lp\xi,\rho\rp:h(x) > 0\right\}\rp\le\omega_k\lp1-\alpha\rp\rho^k\quad and,\quad e^{\lp\Lambda+c^*k\rp\rho}\le 1+\alpha.
\end{equation} 
Suppose also that, for some constant $\Gamma>0$
\begin{equation}
\|V\|\lp B_g(\xi,2\rho)\rp\le\Gamma\rho^k.
\end{equation}
Then there are constants $\beta:=\beta\lp k,\alpha,r_0,b\rp\in]0,\frac12[$ and $C:=C\lp k,\alpha,r_0,b\rp>0$ such that
\[
\int_{\B{\xi}{2\rho}}h\dv\le C\rho\int_{\B{\xi}{\rho}}\lmo \nabla^S_g h\rmo_g\dv.
\]
Here $H_g$ is as in Proposition \ref{firstvariationrepresentation}.
\end{Res}
Finally, the Sobolev inequality follows from Theorem \eqref{FWMI} again in the special case in which $V=\rv{S}{\theta}$ is a rectifiable varifold such that the mean vector field belongs to a certain Lebesgue space and a standard covering argument (c.f. \cite{Simon} Theorem 3.3 pg. 11).
\begin{Res}\label{sobolev}
Let $(M^n,g)$ be a complete Riemannian manifold with bounded geometry, let $V=\rv{S}{\theta}$ a $k$-rectifiable varifold satisfying (AC). Suppose $h\in C^1_0(M)$ non negative, and $\theta\ge1$ $\|V\|$-a.e. in $S$. Then there exists $C:=C(k)>0$ such that
\begin{equation}\label{Eq:IntrinsicSobolevInequality}
\lp\int_S h^{\frac k{k-1}}\dv\rp^{\frac{k-1}k}\le C\int_S\lp\lmo\nabla^S_g h\rmo_g+h\lp\lmo H_g\rmo_g-c^*k\rp\rp\dv.
\end{equation}
Here $H_g$ is as in Proposition \ref{firstvariationrepresentation}.
\end{Res}


\subsection{Structure of the paper}
 In Section $2$ the definition of general varifold of \cite{Allard} is extended to a complete Riemannian manifold as well as the first variation of them. In Section $3$ Theorem \ref{FWMI} is proved, by testing the first variation of a general varifold with a suitable radial deformation field. Here the bounded geometry assumptions of the ambient manifold, play an essential role. In Section $3$ are also proved several results concerning the monotonicity of the density ratio under $L^p$-type assumptions on the weak mean curvature field $H_g$ of the varifold. Finally, in Section 4 Theorem \ref{poincare} and Theorem \ref{sobolev} are proved, as a consequence of  Theorem \ref{FWMI} and a standard covering argument.

\section{Intrinsic Riemannian Theory of General Varifolds}
\subsection{General Varifolds} 
 
 Now we introduce the notations and concepts relative to varifolds that we need to make a Riemannian intrinsic theory of varifolds. In this respect we closely follow \cite{Allard}, \cite{Simon}, \cite{DeLellis}, \cite{NardulliOsorioIMRN}. In what follows $V$ will always denote a varifold and $dv_g$ the Riemannian measure of $(M^n,g)$.

\begin{definition}\label{defgeneralvarifold}
For any $k,n\in\N$, $n\ge2$, $1\le k\le n-1$, let $M^n$ a $n$-dimensional  manifold. We say that $V$ is a \emph{$k$-dimensional varifold} in $M$, if $V$ is a nonnegative, real extended valued, Radon measure on $G_k(M^n)$ the Grassmannian manifold whose underlying set is the union of the sets $\GRKX$, where $\GRKX$ denotes the set of $k$-dimensional subspaces of $T_xM^n$, as $x$ varies on $M^n$  (compare with section $2.6$ of \cite{Allard}). For every $k\in\{1,...,n-1\}$, we define $\V_k(M^n)$ to be the space of all $k$-dimensional varifolds on $M^n$ endowed with the weak topology induced by $C^0_c(G_k(M^n))$, say the space of continuous compactly supported functions on $G_k(M^n)$ endowed with the compact open topology.  
\end{definition}

\begin{definition}
Let $V\in\mathbf{V}_k(M^n)$, $g$ is a Riemannian metric on $M^n$, we say that the nonnegative Radon measure on $M^n$, $||V||$ is the \textit{weight} of $V$, if $||V||=\pi_{\#}(V)$, here $\pi$ indicates the natural fiber bundle projection $\pi:G_k(M^n)\rightarrow M^n$, $\pi:(x, S)\mapsto x$, for every $(x,S)\in G_k(M^n)$, $x\in M^n$, $S\in\GRKX$,
 $$||V||(A):=V(\pi^{-1}(A)).$$ 
\end{definition} 

\begin{remark} 
Recall that $\pi$ is a proper map because the fibers of the fiber bundle $G_k(M^n)\stackrel{\pi}{\to}M^n$ are compact.
\end{remark}

As the reader has noticed, an abstract varifold can be a quite strange object, because it is hard to work with Borel sets on $G_k(M^n)$ in an operative way (here operative way has to be understate in a sense to be specified later in this section); but in the sequel, we also define the \emph{weight of $V$} which is a Radon measure on $\Sigma$ obtained from $V$ by ignoring the fiber variable. The next theorem illustrates how to "simplify" a varifold in an operative way. This result is a direct application of a well known \textit{disintegration Theorem}, which can be found in \cite{Ambrosio} Theorem $2.28$. However the following is an adaptation of it into the context of our intrinsic varifolds.\\

\begin{theorem}[Disintegration Theorem for Varifolds]\label{DVThm}
Let $(M^n,g)$ a $n$-dimensional Riemannian manifold, let $V\in\mathbf{V}_k(M^n)$, and $\pi:G_k(M^n)\to M$ be the canonical projection onto $M^n$. Then there exists a family of Radon measures $\{\pi_x\}_{x\in M^n}$ such that, the map $x\mapsto\pi_x$ is $||V||$-measurable and the following relations are satisfied
\begin{align}
&\pi_x(B):=\lim_{r\downarrow 0}\frac{V(B_g(x,r)\times B)}{||V||(B(x,r))},\ \quad for\ all\ B\in \mathcal{B}(Gr(k, T_xM^n)),\label{Eq.1DThm}\\
& \pi_x\left(Gr(k,T_xM^n)\setminus\{S:S\subset Tan(M^n,x)\}\right)=0, and\ \pi_x(Gr(k,T_xM^n))=1,\label{Eq.2DThm}\\
&f(x,\cdot)\in L^1(Gr(k,T_xM^n),\pi_x),\ for\ ||V||-a.e.\ x\in M^n,\label{Eq.3DThm}\\
&x\mapsto\int_{Gr(k,T_xM^n)}f(x,S)d\pi_x(S)\in L^1(M^n,||V||),\text{is}\ ||V||-\text{measurable},\label{Eq.4DThm}\\
&\int_{G_k(M^n)}f(x,S)dV(x,T)=\int_{M^n}\left(\int_{Gr(k,T_xM^n)}f(x,S)d\pi_x(S)\right)d||V||(x)\label{Eq.5DThm},
\end{align}
for any $f\in L^1(G_k(M^n),V)$. Moreover, if $\pi_x'$ is any other $||V||$-measurable map satisfying \eqref{Eq.4DThm} and \eqref{Eq.5DThm} for every bounded Borel function with compact support and such that $x\mapsto\pi_x'(Gr(k,T_xM^n))\in L^1(M^n,||V||)$, then $\pi_x=\pi_x'$ for $||V||$-a.e. $x\in M^n$.
\end{theorem}
\begin{corollary}
Let $V\in\mathbf{V}_k(M^n)$, with the same notation of the Theorem \ref{DVThm}, the equality
\[
V=||V||\otimes\pi_x,
\]
holds.
\end{corollary}
\subsection{The First Variation of a Varifold}
According to \cite{Allard}, we associate to each varifold $V$ a vector valued distribution, which depends on the metric $g$ of the ambient space, called the \textit{first variation of $V$}. If this first variation is (as Radon measure measure) absolute continuous with respect to weight measure $\|V\|$ in analogy with the smooth case we define the generalized mean curvature vector $H_g$.
\begin{definition}\label{firstvariation}
Let $(M^n,g )$ a $n$-dimensional Riemannian manifold with Levi-Civita connection $\nabla$, $\X^1_c(M)$ the set of differentiable vector fields on $M$ and $V\in\V_k(M)$ a $k$-dimensional varifold ($k\le n$). We define the first variation of $V$ along the vector field $X\in\X_c^1(M)$ as
\[
\delta V(X):=\int_{G_k(M)}\Ll\nabla X(x)\circ\pr{T},\pr{T}\Rl_g dV(x,T),
\]
where the inner product in the integrand is the one defined in $\operatorname{Hom}\lp T_xM,T_xM\rp$, and $\nabla X:\X(M)\to \X(M)$ such that $\nabla X(Y):=\nabla_YX$. 
\end{definition}
Let $\{\tau_1,\dots,\tau_k\}$ an orthonormal basis of $
T\in\GRKX$ for $x$ given, and $\{\tau_1,\dots,\tau_k,\tau_{k+1},\dots\tau_n\}$ the completion to an orthonormal basis for $T_xM$, then for given $X\in\X^1_c(M)$,
\begin{align*}
\delta V(X)
&=\int_{G_k(M)}\trace\lp\lp\nabla X(x)\circ\pr{T}\rp^*\circ\pr{T}\rp dV(x,T)\notag\\
&=\int_{G_k(M)}\sum_{i=1}^k\Ll\tau_i,\nabla_{\tau_i}X(x)\Rl_gdV(x,T),
\end{align*}
so, we have that, for all $X\in\X^1_c(M)$, the definition \eqref{firstvariation} is equivalent to say,

\begin{equation}\label{generalfirstvariation}
\delta V(X)=\int_{G_k(M)}\divergence_TX(x)dV(x,T),
\end{equation}

where, for given $x$, and $\left\{\tau_1,\dots,\tau_k\right\}$ an orthonormal basis of a fixed $T\in\GRKX$,
\[
\divergence_T X(x)=\sum_{i=1}^k\Ll\tau_i,\nabla_{\tau_i}X(x)\Rl_g.
\]
Then \eqref{generalfirstvariation} give us formula with more geometric meaning than the merely definition, on the other hand, by analogy with the smooth case, we desire to related in some way the first variation of a varifold with a "mean curvature" vector, to do so, let us first analyze the total variation of the first variation.\\
\begin{definition}\label{boundedfirstvariation}
Let $V\in\V_k(M)$ with $(M,g)$ a $n$-dimensional Riemannian manifold, we say that $V$ has locally bounded first variation in $U\subset M$ open set, if for all $W\subset\subset U$ open set, there exists a constant $C:=C(W)$ such that
\[
\lmo \delta V(X)\rmo\le C\|X\|_{L^{\infty}\lp W,V\rp},
\]
for all $X\in\X_c^0(W)$.
\end{definition}
\begin{remark}\label{extensionoffirstvariation}
Notice that the definition of bounded first variation is valid for each $X\in\X_c^0(M)$ while the first variation, $\delta V$, is defined on $\X_c^1(M)$. However, defining an extension $\widetilde{\delta V}:\X_c^0(M)\to\R$ of $\delta V$ as
\[
\widetilde{\delta V}(X):=\lim_{\varepsilon\downarrow 0}\delta V(X_{\varepsilon}),
\]
where $(X^1,\dots,X^n)=X\in\X^0_c(M)$, and $X_{\varepsilon}\in\X^1_c(M)$ is a "$C^1$ approximation" of $X$ defined as follows:\\
Let $\{(\Phi_i,U_i)\}_{i\in\Lambda}$ an atlas of $M$, since we are interested in compactly vector fields we can choose $\Lambda=\{1,\dots, N\}$ such that $\{U_i\}_{i=1}^N$ is an open covering for $\spt X$, and consider $\{\psi\}_{i=1}^{\ell}$ a partition of unity subordinate to $\{U_i\}_{i=1}^N$, then
\[
X_{\varepsilon}(p):=\sum_{i=1}^{\ell}\lp\sum_{j=1}^N\lp\psi_j \left.X^i\rmo_{U_i}\star\varphi^{\varepsilon}\rp(p)\rp\frac{\partial}{\partial x_i},
\]
where $\varphi^{\varepsilon}$ is an standard approximation of the identity in $\R^n$, and $\star$ denotes the usual convolution.\\
Clearly $X_{\varepsilon}$ is independent of the choice of charts and defined on $\spt X$. Furthermore by standard theory of convolutions, we know that,
\[
\psi_j \left.X^i\rmo_{U_i} \star\varphi^{\varepsilon}\to\psi_j\left.X^i\rmo_{U_j}
\]
uniformly in compacts, when $\varepsilon\to 0$. Hence
\[
\|X-X_{\varepsilon}\|_{L^{\infty}}\to0,\quad when\ \varepsilon\to 0,
\]
and $X_{\varepsilon}\in\X_c^1(M)$.\\
On the other hand, given $X^1_{\varepsilon},X^2_{\varepsilon}\in\X_c^1(M)$ as above such that $X^1_{\varepsilon}\to X$ and $X^1_{\varepsilon}\to X$, assume that $V$ has bounded first variation on $M$, then
\[
\lmo\delta V(X^1_{\varepsilon}-X^2_{\varepsilon})\rmo\le C\|X^1_{\varepsilon}-X^1_{\varepsilon}\|_{L^{\infty}}\to 0\quad when\ \varepsilon\to 0,
\]
which implies that 
\[
\widetilde{\delta V}:\X^0_c(M)\to\R
\]
is well defined and $\widetilde{\delta V}\equiv\delta V$ on $\X^1_c(M)$. By abuse os notation we will identify $\widetilde{\delta V}$ with $\delta V$.
\end{remark}
\begin{proposition}\label{firstvariationrepresentation}
Let $V\in\V_k(M^n)$ with locally bounded first variation in $(M^n,g)$. Then the total variation $\|\delta V\|$ is a Radon measure. Furthermore, there exist a $\|V\|$-measurable function $H_g:M\to TM$, and $Z\subset M$ with $\|V\|(Z)=0$, such that
\begin{equation*}
\delta V(X)=-\int_{M^n}\Ll X,H_g\Rl_gd\|V\|+\int_{M^n}\Ll X,\nu\Rl_gd\|\delta V\|_{sing},
\end{equation*}
where $\nu$ is a $\|\delta V\|$-measurable function with $\abs{\nu(x)}_g=1$, and $\|\delta V\|_{sing}=\|\delta V\|\llcorner Z$. We call $H_g$ the generalized mean curvature vector of $V$.  
\end{proposition}
\begin{proof}
Since by hypothesis $V$ has locally bounded first variation, for all open set $W\ssubset M$
\[
\|\delta V\|(W)=\sup\{\ |\delta V(X)|:X\in\X^0_c(W), |X|\le 1\}\le C<\infty.
\]
Then by the Riesz Representation Theorem  for $\X_c^0(M)$ (where we are identifying $\widetilde{\delta V}$ with $\delta V$) there exists $\nu: M\to TM$ $\|\delta V\|$-measurable function with $|\nu(x)|_g=1$ $\|\delta V\|$-a.e. and for all $X\in\X_c^0(M)$
\[
\delta V(X)=\int_M\Ll X,\nu\Rl_gd\|\delta V\|,
\]
and $\|\delta V\|$ is a Radon measure. This last assertion prove the first part of the theorem as in the Euclidean case. To prove the second part of the theorem we need to make use of the fact that the Vitali symmetry property holds for every metric balls of $(M^n,g)$. This guarantees that we can apply the  Lebesgue Differentiation Theorem to space $B_g(p,r_p)$. Thus, by the  Lebesgue Differentiation Theorem we get 
\[
D_{\|V\|}\|\delta V\|=\lim_{\rho\downarrow 0}\frac{\|\delta V\|\lp B_g(x,\rho)\rp}{\|V\|\lp B_g(x,\rho)\rp}
\]
exists $\|V\|$-a.e and (writing $-H_g(x):=D_{\|V\|}\|\delta V\|(x)\nu(x)$) we are lead to the following formula
\[
\delta V(X)=-\int_M\Ll X(x),H_g(x)\Rl_gd\|V\|(x)+\int\Ll X(x),\nu(x)\Rl_gd\|\delta V\|_{sing},
\]
where $\|\delta V\|_{sing}=\|\delta V\|\llcorner \left\{x:D_{\|V\|}\|\delta V\|=+\infty\right\}$.
\end{proof}

\begin{remark}\label{allardconditionequivalence}
Let $(M^n,g)$ be an $n$-dimensional Riemannian manifold, let $V\in\V_k(M)$ such that for some $C>0$, $p>k$, $0<\rho<inj_{\xi}(M,g)$ for fixed $\xi\in M$ 
\begin{equation}\label{allardcondition}
\delta V(X)\le C\lp\int_{B_g(\xi,\rho)}|X|^{\frac{q}{q-1}}d\|V\|\rp^{\frac{q-1}{q}}.
\end{equation}
Then, by H\"older inequality we have that 
\[
\lmo\delta V(X)\rmo\le C_1\|X\|_{L^{\infty}(\|V\|,\spt X\cap B_g(\xi,\rho))},
\]
where $C_1(\xi,\rho):=C\lp\|V\|\lp \spt\ X\cap B_g(\xi,\rho)\rp\rp^{\frac{q-1}{q}}$. Then $V$ has locally bounded first variation, furthermore,  $\|\delta V\|$ is a Radon measure on $M$.
Roughly speaking, if we  assume that $\spt X\subset A$, for $A$ a given $\|V\|$-negligible set, and take the supremum over all such $X$'s with $\|X\|_{L^{\infty}(\|V\|)}=1$ then by \eqref{allardcondition} we get that 
\[
\|\delta V\|(A)\le C\lp\|V\|(A)\rp^{\frac{q-1}{q}}=0,
\]
which implies that $\|\delta V\|\ll\|V\|$. In fact, technically to prove the preceding equation we need to be more carefull, i.e., we need to take a sequence of open sets $A\subset U_j\subseteq B_g(\xi,\rho),\forall j$ such that $U_j\to A$, then using \eqref{allardcondition} we have that 
$$\delta V(X)\le C\lp\int_{U_j}|X|^{\frac{q}{q-1}}d\|V\|\rp^{\frac{q-1}{q}},\forall X\in\X_c^1(U_j).$$ Then taking limits, it easy to check that 
 \[
\|\delta V\|(A)=\lim_{j\to\infty}\|\delta V\|(U_j)\le\lim_{j\to\infty}C\lp\|V\|(U_j)\rp^{\frac{q-1}{q}}=0.
\]
Hence, the Radon-Nikodym Theorem gives the existence of a $L^1(M^n,\|V\|)$ function $H_g$ such that 
\begin{equation}\label{interior1stvar}
\delta V(X)=\int_{G_k(M)} div_{S}X(y)dV(y,T)=-\int_M\Ll H_g(y),X(y)\Rl_g d\|V\|(y).
\end{equation}
On the other hand, \eqref{allardcondition} means that $\delta V(X)$ is a bounded linear functional on the space $L^q_{loc}(\Gamma(TM), \|V\|)$ of the $L^q_{loc}$ vector fields on $M$ with respect to $\|V\|$. So by well known theorems of measure theory (compare Theorem $6.16$ of \cite{RudinRealAndComplex}) stating that the topological dual of $L^q$ is $L^p$, i.e., $(L^q)^*=L^p$ we have that there exists a function $f\in L^p$ such that $\delta V(X)=\int_{B_g(\xi,\rho)}\left\langle f,X\right\rangle_g$ and the best constant $C>0$ in \eqref{allardcondition} is given by $||f||_{L^p}$.
Combining \eqref{allardcondition} and \eqref{interior1stvar} we get that $f=H_g$ and so that $H_g\in L^p_{loc}(\Gamma(TM), \|V\|)$. Finally, notice that, for all $X\in\X^1_c(M)$ (again using the H\"older inequality)
\[
\delta V(X)\leq\|H_g\|_{L^p(\|V\|,M)}\|X\|_{L^{\frac{q}{q-1}}(\|V\|,M)}.
\]
 
At this stage we have proved that \eqref{allardcondition} implies that $V$ has locally bounded first variation, no singular part and the generalized mean curvature vector field $H_g$, satisfies an $L^p$ condition. 

\end{remark}
\begin{definition}[\hypertarget{$(AC)$}{Allard's Conditions}]\label{defallardcondition}
Given $V$ a $k$-dimensional varifold, we say that $V$ satisfies an \textit{Allard's type condition} for the generalized mean curvature, if $V$ satisfies \eqref{allardcondition}. From now on, denoted as \hyperlink{$(AC)$}{$(AC)$}.
\end{definition}
This kind of conditions will play a key role in the section \eqref{monotonicity} when we study the monotonicity behaviour of the density ratio. 


 \section{Monotonicity and Consequences}
 
The aim of this section is to obtain information about $V$ from their first variation $\delta V(X)$, for an appropriate choice of $X\in\mathfrak{X}^1_0(U)$.  This kind of result is known as \textbf{"Monotonicity Formula"}, and is the step zero in any interior regularity theory.\\


\subsection{Weighted Monotonicity Formulae for Abstrac Varifolds }\label{monotonicity}
 Let $(M,g)$ a Riemannian manifold with Levi-Civita connection $\nabla$  and $V\in\mathbf{V}_k(M)$ satisfying \hyperlink{$(AC)$}{$(AC)$}, fix $\xi\in M$ and assume $inj_{\xi}(M,g)>0$. Choose $r_0>0$ such that $r_0<inj_\xi(M,g)$, and, for fixed $\varepsilon>0$  let $\gamma_{\varepsilon}\in C^1_c(]-\infty,1[)$, such that
\[
\gamma_{\varepsilon}(y):=
\begin{cases}
1\quad if\ y\leq\varepsilon,\\
0\quad if\ y>1,
\end{cases}
\quad
and
\quad
 \gamma'_{\varepsilon}(y)<0\ if\ \varepsilon <y<1,
\]

Then we can consider the radial perturbation 
\[
\tilde{X}_{s,\varepsilon}(x)=\left(\gamma_{\varepsilon}\left(\frac{u(x)}{s}\right)(u\nabla u)\right)(x),\quad for\ 0<|s|<r_0,
\]
where $u(x)=r_\xi(x)=dist_{(M,g)}(x,\xi)$.  Now, let $y\in M$ given, and $T\in\GRKY$. Let $\{e^T_1,\dots,e^T_k\}$ an orthonormal basis for $S$, then
\begin{equation}\label{eqch1.1}
\begin{split}
\divergence_{T}\tilde{X}_{s,\varepsilon}
&=\sum_{i=1}^{k}\Ll e_i^T,\nabla_{e_i^T}\gamma_{\epsilon}\left(\frac{u}{s}\right)(u\nabla u)\Rl_g\notag\\
&=\gamma'_{\varepsilon}\left(\frac{u}{s}\right)\frac{u}{s}\sum_{i=1}^k\Ll e_i^T,\nabla u\Rl_g^2+\gamma_{\epsilon}\left(\frac{u}{s}\right)\sum_{i=1}^k\Ll e_i^T,\nabla_{e_i^T}\left(u\nabla u\right)\Rl_g\notag\\
&=\gamma_{\epsilon}\left(\frac{u}{s}\right)\divergence_{T}(u\nabla u)+\gamma'_{\epsilon}\left(\frac{u}{s}\right)\frac{u}{s}\left|\nabla^{\top} u\right|_g^2,
\end{split}
\end{equation}
where $\nabla^{T} u$ is the orthogonal projection of $\nabla u$ onto $S$. Let $\nabla^{T^{\bot}} u$ denote the orthogonal projection of $\nabla u$ onto $T^{\bot}$, then,
\[
|\nabla^{T} u|_g^2+|\nabla^{T^{\bot}} u|_g^2=|\nabla u|_g^2=1,
\]
since $u$ is a distance function.
Therefore 
\[
\divergence_{T}\tilde X_{s,\varepsilon}=\gamma_{\epsilon}\left(\frac{u}{s}\right)
\divergence_{T}(u\nabla u)+\gamma'_{\epsilon}\left(\frac{u}{s}\right)\frac{u}{s}-\gamma'_{\epsilon}\left(\frac{u}{s}\right)\frac{u}{s}\left|\nabla^{S^{\bot}}u\right|^2_g.
\]
Although this choice is enough to get many useful information, let us consider a general case, which will be used in the sequel. Let $h\in C^1(U)$ a non-negative function, and consider
\begin{equation}\label{radialperturbation}
X_{s,\varepsilon}(y):=h(y)\tilde{X}_s(y)=h(y)\left(\gamma_{\varepsilon}\left(\frac{u(y)}{s}\right)(u\nabla u)\right)(y),\quad for\ 0<|s|<r_0.
\end{equation}
Then, if $\{e_1^S,\dots,e^T_k\}$ is as above,
\begin{equation}\label{eqch1.2}
\begin{split}
\divergence_TX_{s,\varepsilon}(y)
&=\sum_{i=1}^k\Ll e^T_i,\nabla_{e^T_i}h(y)\tilde{X}_s(y)\Rl_g\\
&=h(y)\sum_{i=1}^k\Ll e^T_i,\nabla_{e^T_i}\tilde{X}_s(y)\Rl_g+\sum_{i=1}^k\Ll \lp \nabla_{e^T_i}h(y)\rp e^T_i,\tilde{X}_s(y)\Rl_g\\
&=h(y) \divergence_T\tilde{X}_{s,\varepsilon}+\Ll\nabla^Th(y),\tilde{X}_s(y)\Rl_g.
\end{split}
\end{equation}
By the definition of the first variation, we know that

\begin{align*}
\delta V(X_{s,\varepsilon}):&=\int_{G_k(U)}\divergence_TX_{s,\varepsilon}(y)dV(y,T)\\
&=\int_{G_k(U)}h(y)\divergence_T\tilde{X}_s(y)dV(y,T)+\int_{G_k(U)}\Ll\nabla^T h(y),\tilde{X}_s(y)\Rl_gdV(y,T).
\end{align*}

Then, replacing \eqref{eqch1.1} in \eqref{eqch1.2} and the information above, we have

\begin{equation}\label{eqch1.3}
\begin{split}
\delta V(X_{s,\varepsilon})
& =  \int_{G_k(U)}h(y)\gamma_{\varepsilon}\left(\frac{u(y)}{s}\right)\divergence_{T}(u\nabla u)(y)dV(y,T) \\
& + \int_{G_k(U)}\frac{h(y) u(y)}{s}\gamma_{\varepsilon}'\left(\frac{u(y)}{s}\right)dV(y,T)\\
& -  \int_{G_k(U)}\frac{h(y) u(y)}{s}\gamma_{\varepsilon}'\left(\frac{u(y)}{s}\right)\left|\nabla^{T^{\bot}}u\right|^2_g dV(y,T)\\
& + \int_{G_k(U)}\gamma_{\varepsilon}\left(\frac{u(y)}{s}\right)\Ll\nabla^T h(y),\lp u\nabla u\rp(y)\Rl_gdV(y,T)
\end{split}
\end{equation}

If we compare \eqref{eqch1.3} with the Euclidean case (see Cf. \cite{Simon} (4.14) and (4.24)) we note the lack of a dimensional term, then, in order to have an intrinsic result, we need to compare it in some way with the Euclidean case. To do this we use the Rauch's comparison theorem, applied as in Lemma 3.6 of \cite{Hoffman1974sobolev}, which states:
\begin{lemma}\label{HSLemma}
Let $(M,g)$ be a complete Riemannian manifold, with Levi-Civita connection $\nabla$, let $b\in\R$ such that $Sec_g\leq b$,  assume  $br_0<\pi$. Then
\[
\divergence_T(u\nabla u)(x)\geq ku(x)\cot_b(u(x)),
\]
for all $x\in B_g(\xi,r_0)$.
\end{lemma}
\begin{remark}\label{comentsHSlemma}
Before to continue, let us give some comments on this Lemma. First to all, recall that $\cot_b(s)$ is defined as 
\[
\cot_b(s)=\frac{cs_b(s)}{sn_b(s)},
\]
where $sn_b(s)$ is the unique solution to 
\begin{equation*}
\begin{cases}
x''(s)+bx(s)&=0\\
x(0)&=0\\
x'(0)&=1,
\end{cases}
\end{equation*}
and $cs_b(s)=sn_b'(s)$.\\

Then, we have three cases
\begin{itemize}
\item[Case 1:] assume $b=0$, then \\\[
s\cot_b(s)=1.
\]
\item[Case 2:] Assume $b>0$, then \\\[
s\cot_b(s)=s\sqrt{b}\cdot\frac{\cos\lp\sqrt{b}s\rp}{\sin\lp\sqrt{b}s\rp}.
\]
An asymptotic analysis when $s\to 0$ gives
\begin{equation}\label{asymptoticcot1}
s\cot_b(s)
\sim1-O(s^2)
\end{equation}
\item[Case 3:] Assume $b<0$, then \\\[
s\cot_b(s)=\sqrt{|b|}s\cdot\frac{\cosh\lp\sqrt{|b|}s\rp}{\sinh\lp\sqrt{|b|}s\rp}.
\]
As in the previous case, an asymptotic analysis when $s\to 0$ gives
\begin{equation}\label{asymptoticcot1}
s\cot_b(s)
\sim1+O(s^2).
\end{equation}

\end{itemize}
\end{remark}

\begin{remark}\label{geometric conditions}
Notice that under the conditions of  Lemma \ref{HSLemma} we can compare the divergence of the radial field $u\nabla u$ in the ambient manifold $M^n$ with the same quantity calculated in a space-form, so, is natural to guess that this is the framework to get some "monotonicity" behaviour (inspired in the results from the euclidean case). Furthermore, in view of Remark \ref{comentsHSlemma} above, we expect to recovery in some sense the euclidean case, this is why from now we will be in the setting of the Lemma \ref{HSLemma}, i.e.
\end{remark}
\begin{definition}[ \hypertarget{$(BG)$}{Bounded Geometry}]\label{Def:GeometricConditions}
We say that  a complete Riemannian manifold $(M^n,g)$ has \textit{bounded geometry} (or satisfy \hyperlink{$(BG)$}{$(BG)$}), if $Sec_g<b$ for some constant $b\in\R$ and for every $\xi\in M$ there is $r_0$  such that $0<r_0<inj_{\xi}(M,g)$ and $r_0b<\pi$. 
\end{definition}

In view of the Remarks \eqref{geometric conditions} and \eqref{comentsHSlemma} we can define $c(s):=c(s,b)$ as:
\[
c(s)=
\begin{cases}
s\sqrt{b}\cot\lp\sqrt{b}s\rp &b>0\\
0&b\le0,
\end{cases}
\]
where $b\in\R$. Then, by the Lemma \eqref{HSLemma},
\[
\divergence_T\lp(u\nabla u)(x)\rp\geq kc(u(x),b),
\]
for all $x\in B_g(\xi,r_0)$, furthermore we can assume $r_0\cot_b(r_0)>0$, since we are interested in small geodesic balls around $\xi$.\\
 Then, substituting in \eqref{eqch1.3} we have
\begin{align*}
\delta V(X_{s,\varepsilon}) & \geq  \int_{G_k(M)}kch(y)\gamma_{\varepsilon}\left(\frac{u(y)}{s}\right)dV(y,T)\notag\\
& + \int_{G_k(M)}\gamma_{\varepsilon}\left(\frac{u(y)}{s}\right)\Ll\nabla^T h(y),\lp u\nabla u\rp(y)\Rl_gdV(y,T)\notag\\
& + \int_{G_k(M)}\frac{h(y) u(y)}{s}\gamma_{\varepsilon}'\left(\frac{u(y)}{s}\right)dV(y,T)\notag\\
& -  \int_{G_k(M)}\frac{h(y) u(y)}{s}\gamma_{\varepsilon}'\left(\frac{u(y)}{s}\right)\left|\nabla^{T^{\bot}}u\right|^2_gdV(y,T).
\end{align*}

On the other hand, since$c(s)$ is decrrasing,  $c(u(x))\ge c(s)\ge c(r_0)$ whenever $0\le u(x)\le s\le r_0=r_0(b)$. Hence, \\

\begin{align*}
-\int_{G_k(M)}\left[hk\gamma_{\varepsilon}\left(\frac{u}{s}\right)+\frac{h u}{s}\gamma_{\varepsilon}'\left(\frac{u}{s}\right)\right]dV
&\geq   \int_{G_k(M)}-\frac{h u}{s}\gamma_{\varepsilon}'\left(\frac{u}{s}\right)\left|\nabla^{T^\bot}u\right|^2_gdV\\
& +  \int_{G_k(M)}\gamma_{\varepsilon}\lp\frac{u}{s}\rp\Ll\nabla^T h,(u\nabla u)\Rl_gdV\\
& +  (c(s)-1)\int_{G_k(M)}hk\gamma_{\varepsilon}\left(\frac{u}{s}\right)dV\\
& - \delta V(X_{s,\varepsilon}),
\end{align*}

now, dividing by $s^{k+1}$,

\begin{align*}
-\int_{G_k(M)}h(y)I(s)dV(y,T)
& \geq  \int_{G_k(M)}\frac{h(y)}{s^{k}}\frac{\partial}{\partial s}\lp\gamma_{\varepsilon}\left(\frac{u(y)}{s}\rp\right)\left|\nabla^{S^{\bot}}u\right|^2_gdV(y,T)\\
& + \frac{1}{s^{k+1}}\lp\int_{G_k(M)}\gamma_{\varepsilon}\lp\frac{u(y)}{s}\rp\Ll\nabla^T h(y),(u\nabla u)(y)\Rl_gdV(y,T)\rp\\
& + \frac{(c(s)-1)}{s}\frac{k}{s^k}\int_{G_k(M)}h(y)\gamma_{\varepsilon}\left(\frac{u(y)}{s}\right)dV(y,T)\\
& -  \frac{\delta V(X_{s,\varepsilon})}{s^{k+1}},
\end{align*}

where
\begin{align*}
I(s):&=\frac{k}{s^{k+1}}\gamma_{\varepsilon}\left(\frac{u(y)}{s}\right)+\frac{ u(y)}{s^{k+2}}\gamma_{\varepsilon}'\left(\frac{u(y)}{s}\right)\\
&=-\frac d{ds}\lp\frac 1{s^k}\rp\gamma_{\varepsilon}\lp\frac {u(y)}s\rp-\frac1{s^k}\frac d{ds}\lp\gamma_{\varepsilon}'\lp\frac{u(y)}{s}\rp\rp\\
&=-\frac d{ds}\lp \frac1{s^k}\gamma_{\varepsilon}\lp\frac{u(y)}s\rp\rp.
\end{align*}

Differentiating under the sign of integral we have

\begin{align*}
\frac{d}{ds}\lp\frac{1}{s^k}\int_{G_k(M)}h\gamma_{\varepsilon}\lp\frac{u}{s}\rp dV\rp & \geq  \int_{G_k(M)}\frac{h}{s^{k}}\frac{\partial}{\partial s}\lp\gamma_{\varepsilon}\left(\frac{u}{s}\rp\right)\left|\nabla^{T^{\bot}}u\right|^2_gdV\\
& + \frac{1}{s^{k+1}}\lp\int_{G_k(M)}\gamma_{\varepsilon}\lp\frac{u}{s}\rp\Ll\nabla^T h,(u\nabla u)\Rl_gdV\rp\\
& + \frac{(c(s)-1)}{s}\frac{k}{s^k}\int_{G_k(M)}h\gamma_{\varepsilon}\left(\frac{u}{s}\right)dV\\
& -  \frac{\delta V(X_{s,\varepsilon})}{s^{k+1}}.
\end{align*}

Now, by Theorem \eqref{DVThm}, 
\[
\int_{G_k(M)}h(y)\gamma_{\varepsilon}\lp\frac{u(y)}{s}\rp dV(y,T)=\int_{M}h(y)\gamma_{\varepsilon}\lp\frac{u(y)}{s}\rp d\|V\|(y)
\]
and
\[
\int_{G_k(M)}h(y)k\gamma_{\varepsilon}\left(\frac{u(y)}{s}\right)dV(y,T)=\int_{M}h(y)k\gamma_{\varepsilon}\left(\frac{u(y)}{s}\right)d\|V\|(y).
\]
Thus, 
\begin{equation}\label{eqch1.4}
\begin{split}
\frac{d}{ds}\lp\frac{1}{s^k}\int_{M}h\gamma_{\varepsilon}\lp\frac{u}{s}\rp d\|V\|\rp 
& \geq  \int_{G_k(M)}\frac{h}{s^{k}}\frac{\partial}{\partial s}\lp\gamma_{\varepsilon}\left(\frac{u}{s}\rp\right)\left|\nabla^{T^{\bot}}u\right|^2_gdV\\
& + \frac{1}{s^{k+1}}\lp\int_{G_k(M)}\gamma_{\varepsilon}\lp\frac{u}{s}\rp\Ll\nabla^T h,(u\nabla u)\Rl_gdV\rp\\
& + \frac{(c(s)-1)}{s}\frac{k}{s^k}\int_{G_k(M)}h\gamma_{\varepsilon}\left(\frac{u}{s}\right)d\|V\|\\
& -  \frac{\delta V(X_{s,\varepsilon})}{s^{k+1}}.
\end{split}
\end{equation}
\normalsize
\begin{remark}\label{behaivourc}
Notice that, by the Remark \eqref{comentsHSlemma} the behavior of $c(s):=c(s,b)$, we have that
\begin{equation*}\label{eqch1c*}
\frac{(c(s)-1)}{s}=-O(s)=O(s),\quad as\ s\to 0^+
\end{equation*}
furthermore, this is decreasing, and
\[
\frac{(c(s)-1)}{s}\geq\frac{(c(r_0)-1)}{r_0}:=c^*_1
\]
\end{remark}
\begin{theorem}[Fundamental Weighted Monotonicity Inequality]\label{wightedmonotonicitytheorem}\label{fundamentalinequality}
Let $(M^n,g)$ be a complete Riemannian Manifold with Levi-Civita connection $\nabla$, satisfying \hyperlink{$(BG)$}{$(BG)$}

such that $r_0\cot_b(r_0)>0$, and  let $V\in\V_k(M^n)$ satisfying \hyperlink{$(AC)$}{$(AC)$},

then for any $0<s<r_0$. There exists a constant $c=c(s,b)\in]0,1[$ such that, if we set $u(x)=r_{\xi}(x)=dist_{(M,g)}(x,\xi)$ we have for all $0<s<r_0$

\begin{equation}\label{eqch1.5}
\begin{split}
\frac{d}{ds}\lp\frac{1}{s^k}\int_{B_g(\xi,s)}h d\|V\|\rp 
&  \geq \frac{1}{s^{k+1}}\lp\int_{G_k(B_g(\xi,s))}\Ll\nabla^T h,(u\nabla u)\Rl_gdV\rp\\
& + \frac{(c(s)-1)}{s}\frac{k}{s^k}\int_{B_g(\xi,s)}hd\|V\|\\
& +  \frac{1}{s^{k+1}}\int_{B_g(\xi,s)}\Ll H,h(u\nabla u)\Rl_gd\|V\|\\
& + \frac{d}{ds}J(s),
\end{split}
\end{equation}
where 
\[
J(s):=\int_{G_k(B_g(\xi,s))}h(y)\frac{\left|\nabla^{T^\bot}u\right|^2_g}{r_{\xi}^k}dV(y,T).
\]

\end{theorem}
\begin{proof}
Let $\{\varepsilon_j\}_{j\in\N}$ a sequence, such that $\varepsilon_j\uparrow 1$ when $j\to\infty$ and 
$\{\gamma_{\varepsilon_j}\}\subset C^1_0(]-\infty,1[)$ a sequence of mollifiers, such that $\gamma_{\varepsilon_j}\to\chi_{]-\infty,1[}$, pointwise from below. Considering $X_{s,\varepsilon_j}$ as in \eqref{radialperturbation} by the previous reasoning (see \eqref{eqch1.4}) we have for every $j\in\N$
\small
\begin{align*}
\frac{d}{ds}\lp\frac{1}{s^k}\int_{M}h\gamma_{\varepsilon_j}\lp\frac{u}{s}\rp d\|V\|\rp &
 \geq  \int_{G_k(M)}\frac{h}{s^{k}}\frac{\partial}{\partial s}\lp\gamma_{\varepsilon_j}\left(\frac{u}{s}\rp\right)\left|\nabla^{T^{\bot}}u\right|^2_gdV\\
& + \frac{1}{s^{k+1}}\lp\int_{G_k(M)}\gamma_{\varepsilon_j}\lp\frac{u}{s}\rp\Ll\nabla h,(u\nabla u)\Rl_gdV\rp\\
& + \frac{(c(s)-1)}{s}\frac{k}{s^k}\int_{G_k(M)}h\gamma_{\varepsilon_j}\left(\frac{u}{s}\right)d\|V\|\\
& -  \frac{\delta V(X_{s,\varepsilon})}{s^{k+1}}.
\end{align*}
\normalsize
Letting $j\to\infty$, in virtue of Lebesgue dominated convergence theorem, we have for all $0<s<r_0$,
\small
\begin{equation}\label{eqch1.6}
\begin{split}
\lim_{j\to\infty} \int_Mh(y)\gamma_{\varepsilon_j}\left(\frac{u(y)}{s}\right)d||V||(y)&=\int_Mh(y) \chi_{B_g(\xi,s)}(y)d||V||(x)\\
&=\int_{B_g(\xi,s)}h(y)d||V||(y).
\end{split}
\end{equation}
\normalsize
On the other hand, since $\left|\gamma'_{\varepsilon_j}\left(\frac{u(x)}{s}\right)\right|<0$, for all $j\in\N$ and for all $x\in M$ such that $\varepsilon_j\leq u(x)/s\leq1$,  and $0<s<r_0$. Then
\small
\[
\frac{\partial}{\partial s}\left(\gamma_{\varepsilon_j}\left(\frac{u(x)}{s}\right)\right)=\gamma'_{\varepsilon_j}\left(\frac{u(x)}{s}\right)\left(\frac{-u(x)}{s^2}\right)>0.
\]
\normalsize
Therefore, since $h\geq 0$
\small
\begin{align*}
h(y)\frac{\varepsilon_j^k}{s^k}\frac{\partial}{\partial s}\left(\gamma_{\varepsilon_j}\left(\frac{u(y)}{s}\right)\right)\left|\nabla^{T^{\bot}}u(y)\right|^2_g
&\leq h(y)\frac{\varepsilon_j^k}{u(y)^k}\frac{\partial}{\partial s}\left(\gamma_{\varepsilon_j}\left(\frac{u(y)}{s}\right)\right)\left|\nabla^{T^{\bot}}u(y)\right|^2_g\\
&\leq h(y)\frac{1}{s^k}\frac{\partial}{\partial s}\left(\gamma_{\varepsilon_j}\left(\frac{u(y)}{s}\right)\right)\left|\nabla^{T^{\bot}}u(y)\right|^2_g,
\end{align*}
\normalsize
for all $n\in\N$. Integrating over $G_k(M)$,
\small
\begin{align*}
\int_{G_k(M)}h\frac{\varepsilon_j^k}{s^k}\frac{\partial}{\partial s}\left(\gamma_{\varepsilon_j}\left(\frac{u}{s}\right)\right)\left|\nabla^{T^{\bot}}u\right|^2_gdV
&\leq\int_{G_k(M)}h\frac{\varepsilon_j^k}{u^k}\frac{\partial}{\partial s}\left(\gamma_{\varepsilon_j}\left(\frac{u}{s}\right)\right)\left|\nabla^{T^{\bot}}u(x)\right|^2_gdV\\
&\leq\int_{G_k(M)}h\frac{1}{s^k}\frac{\partial}{\partial s}\left(\gamma_{\varepsilon_j}\left(\frac{u}{s}\right)\right)\left|\nabla^{T^{\bot}}u\right|^2_gdV,
\end{align*}
\normalsize
and making $j\to\infty$, we have in distributional sense

\begin{equation}\label{eqch1.7}
\lim_{j\to\infty}\int_Mh(y)\frac{1}{s^k}\frac{\partial}{\partial s}\left(\gamma_{\varepsilon_j}\left(\frac{u(y)}{s}\right)\right)\left|\nabla^{S^{\bot}}u\right|^2_gdV(y,T)=\frac{d}{ds} J(s)
\end{equation}

Finally, since $V$ satisfies \hyperlink{$(AC)$}{$(AC)$}, we have (see Remark \ref{allardconditionequivalence})
\begin{equation*}
\delta V(X_{s,\varepsilon_j})
=-\int_M h(y)\gamma_{\varepsilon_j}\lp\frac{u(y)}{s}\rp\Ll H_g (y),(u\nabla u)(y)\Rl_gd\|V\|(y),
\end{equation*}
then, taking $X_{s}(y)=h(y)\chi_{B_g(\xi,s)}(y)(u\nabla u)(y)$,
\begin{equation*}
|\delta V(X_{s,\varepsilon_j}-X_s)|
\le\int_M |h(y)| \left|\gamma_{\varepsilon_j}\lp\frac{u(y)}{s}\rp-\chi_{]-\infty,1[}\lp\frac{u(y)}{s}\rp\right|\abs{H_g}_{g}d\|V\|(y).\\
\end{equation*}
By the Lebesgue Dominated Donvergence Theorem, we have $\gamma_{\varepsilon_j}\to\chi_{]-1,1[}$ in $L^1$ sense, then, letting  $1\leq p'<+\infty$ be the conjugate exponent to $1<p\leq+\infty$, by H\"older inequality we have that 
\[
\left\|\gamma_{\varepsilon_j}-\chi_{]-1,1[}\right\|_{L^{p'}(\|V\|,B_g(\xi,s))}\to 0,
\]
for all $s\in]0,r_0[$ and $1\le p'\le+\infty$. Therefore, 

\[
\int_M\|H_g\||h| \left|\lp\gamma_{\varepsilon_j}-\chi_{]-1,1[}\rp\lp\frac{u}{s}\rp\right|d\|V\|\leq\|h\|_{L^{\infty}(\|V\|)}\|H_g\|_{L^p}\|\gamma_{\varepsilon_j}-\chi_{]-1,1[}\|_{L^{p'}}.
\]
Hence,
\small
\begin{equation}\label{eqch1.9}
\lim_{j\to\infty}\delta V(X_{s,\varepsilon_j})=\frac{1}{s^{k+1}}\int_{B_g(\xi,s)}\Ll  H ,h(y)(u\nabla u)(y)\Rl_gd\|V\|(y),
\end{equation}
\normalsize
and then the result follows from \eqref{eqch1.6}, \eqref{eqch1.7},  and \eqref{eqch1.9}
\end{proof}
\begin{corollary}[Weighted Monotonicity Formula]\label{weigthedmonotonicitycorollary1}
Under the hypothesis of the theorem above, if $0<\sigma<\rho<r_0$. Then
\begin{align*}
\frac{1}{\rho^k}\int_{B_g(\xi,\rho)}hd\|V\|-\frac{1}{\sigma^k}\int_{B_g(\xi,\sigma)}hd\|V\|
&\geq -\int_{\sigma}^{\rho}\frac1{s^k}\lp\int_{B_g(\xi,s)}|\nabla^Th|dV\rp\\
&-\int_{\sigma}^{\rho}\frac1{s^k}\lp\int_{B_g(\xi,s)}h\lp|H_g|_g-kc_1^*\rp d\|V\|\rp ds\\
&+\lp J(\rho)-J(\sigma)\rp,
\end{align*}
\normalsize
where $J(s)$ is defined as in Theorem \ref{wightedmonotonicitytheorem}.
\end{corollary}
\begin{proof}
Integrating over $[\sigma,\rho]\subset[0,r_0]$, from Theorem \ref{wightedmonotonicitytheorem} we have,
\begin{align*}
\int_{\sigma}^{\rho}\lp\frac{d}{ds}\lp\frac{1}{s^k}\int_{B_g(\xi,s)}h d\|V\|\rp\rp ds
& \geq \int_{\sigma}^{\rho}\lp\frac{1}{s^{k+1}}\lp\int_{G_k(B_g(\xi,s))}\Ll\nabla^T h,(u\nabla u)\Rl_gdV\rp\rp ds\\
& + \int_{\sigma}^{\rho}\lp\frac{(c(s)-1)}{s}\frac{k}{s^k}\int_{B_g(\xi,s)}hd\|V\|\rp ds\\
& +  \int_{\sigma}^{\rho}\lp\frac{1}{s^{k+1}}\int_{B_g(\xi,s)}\Ll H_g,h(u\nabla u)\Rl_gd\|V\|\rp ds\\
&+\int_{\sigma}^{\rho}\frac d{ds}J(s)ds.
\end{align*}

Applying the Fundamental Theorem of Calculus to the left-hand term, we have

\begin{equation}\label{eqch1.10}
\int_{\sigma}^{\rho}\lp\frac{d}{ds}\lp\frac{1}{s^k}\int_{B_g(\xi,s)}hd\|V\|\rp\rp=\frac{1}{\rho^k}\int_{B_g(\xi,\rho)}hd\|V\|-\frac{1}{\sigma^k}\int_{B_g(\xi,\sigma)}hd\|V\|.
\end{equation}

To estimate the right-hand term, we precede in four steps. For the last term, again from the Fundamental Theorem of Calculus we have,

\begin{equation}\label{eqch1.11}
\int_{\sigma}^{\rho}\frac d{ds}J(s)ds=J(\rho)-J(\rho).
\end{equation}

For the first term, we have

\small
\begin{equation}\label{eqch1.12}
\int_{\sigma}^{\rho}\lp\frac{1}{s^{k+1}}\lp\int_{G_k(B_g(\xi,s))}\Ll\nabla^T h,(u\nabla u)\Rl_gdV\rp\rp ds\\
\geq-\int_{\sigma}^{\rho}\frac1{s^k}\lp\int_{G_k(B_g(\xi,s))}\abs{\nabla^Th}_gdV\rp ds
\end{equation}
\normalsize
Now, estimating the third and  fourth term, (see Remark \ref{behaivourc}) we have 
\small
\begin{equation}\label{eqch1.13}
\begin{split}
\int_{\sigma}^{\rho}\frac{1}{s^{k+1}}
&\lp\int_{B_g(\xi,s)}\Ll H_g,h(u\nabla u)\Rl_gd\|V\|\rp ds+\int_{\sigma}^{\rho}\frac{c(s)-1}{s^{k+1}}\lp\int_{B_g(\xi,s)}hd\|V\|\rp ds\\
&\geq-\int_{\sigma}^{\rho}\frac{1}{s^k}\lp\int_{B_g(\xi,s)}|H_g|_ghd\|V\|\rp ds+\int_{\sigma}^{\rho}\frac{kc_1^*}{s^k}\lp\int_{B_g(\xi,s)}hd\|V\|\rp ds\\
&\geq-\int_{\sigma}^{\rho}\frac1{s^k}\lp\int_{B_g(\xi,s)}h(|H_g(y)|_g-kc_1^*)d\|V\|\rp ds.
\end{split}
\end{equation}
\normalsize
Hence, the result follows from \eqref{eqch1.10}, \eqref{eqch1.11}, \eqref{eqch1.12}, and \eqref{eqch1.13} 
\end{proof}
\begin{remark}\label{monotonicityspecialcase}
Let $h:\spt \|V\|\subset M\to \R$ such that, $h\equiv1$, and $J(s)$ as in Theorem \eqref{wightedmonotonicitytheorem} then, under the hypothesis of  Theorem \ref{wightedmonotonicitytheorem} as a particular case we have that there exists $c\in]0,1[$ such that, fort all $0<s<r_0$

\begin{equation}\label{monotonicityinequality}
\begin{split}
\frac{d}{ds}\lp\frac{\|V\|\lp B_g(\xi,s)\rp}{s^k}\rp & 
\geq \frac{d}{ds}J(s)
 + \frac{(c(s)-1)}{s}\frac{k}{s^k}\int_{B_g(\xi,s)}d\|V\| \\
& +  \frac{1}{s^{k+1}}\int_{B_g(\xi,s)}\Ll H_g,h(u\nabla u)\Rl_gd\|V\|.
\end{split}
\end{equation}
Now, assume that $V$ has locally bounded first variation in $\B{\xi}{2\rho_0}$ for some $\rho_0<r_0$, and let $\Lambda\ge0$ such that
\[
\|\delta V\|(B_g(\xi,\rho))\le\Lambda\|V\|(B_g(\xi,\rho)),\quad \text{for } 0<\rho<\rho_0,
\]
notice that the hypothesis above are much stronger than \hyperlink{$(AC)$}{$(AC)$}, and those implies that for all $X\in\X^0_c(M)$ with $\spt X\subset B_g(\xi,s\rho_0)$,
\[
\delta V(X)=\int_{G_k(M)} div_{S}X(y)dV(y,T)=-\int_M\Ll H_g(y),X(y)\Rl_g d\|V\|(y).
\]
Then is not so hard to see that \eqref{monotonicityinequality} remains valid and $\abs{H_g(y)}_g\le\Lambda$ for $\|V\|$-a.e. $y\in B_g(\xi,\rho_0)$. Then, setting $c_1^*$ as in the Remark \ref{behaivourc},

\begin{align*}
\frac{d}{ds}\lp\frac{\|V\|\lp B_g(\xi,s)\rp}{s^k}\rp 
&\geq \frac d{ds}J(s)+\frac{kc_1^*}{s^k}\int_{B_g(\xi,s)}d\|V\|-\frac{1}{s^k}\int_{B_g(\xi,s)}\frac{|u\nabla u|_g}{s}|H_g|_gd\|V\|\\
&\geq \frac{d}{ds}J(s)-\frac{1}{s^k}\int_{B_g(\xi,s)}\lp|H_g(y)|_g-kc_1^*\rp d\|V\|(y)\\
&\ge \frac{d}{ds}J(s)-\lp\Lambda-kc_1^*\rp\frac{\|V\|\lp B_g(\xi,s)\rp}{s^k}.
\end{align*}
Hence,

\[
\frac{d}{ds}\lp\frac{\|V\|\lp B_g(\xi,s)\rp}{s^k}\rp+\lp\Lambda-kc_1^*\rp\frac{\|V\|\lp B_g(\xi,s)\rp}{s^k}\ge \frac{d}{ds}J(s)>0.
\]

Finally, multiplying by  $e^{(\Lambda-kc_1^*)s}$ we have

\[
e^{(\Lambda-kc_1^*)s}\frac{d}{ds}\lp\frac{\|V\|\lp B_g(\xi,s)\rp}{s^k}\rp+\lp\Lambda-kc_1^*\rp e^{(\Lambda-kc_1^*)s}\frac{\|V\|\lp B_g(\xi,s)\rp}{s^k}\ge0,
\]
then
\begin{equation}\label{eqch1.monotonicity}
\frac{d}{ds}\lp e^{(\Lambda-kc_1^*)s}\frac{\|V\|\lp B_g(\xi,s)\rp}{s^k}\rp\ge0.
\end{equation}
Therfore, the function
\begin{align*}
f:]0,\rho_0[&\to[0,+\infty[\\
s&\mapsto e^{(\Lambda-kc_1^*)s}\frac{\|V\|\lp B_g(\xi,s)\rp}{s^k}
\end{align*}
in non-decreasing.\\
In the sequel we exploit the $L^p$ condition over the generalized mean curvature $H$  to get similar results, concerning to the monotonicity of this ratio.
\end{remark}


\subsection{$L^p$ Monotonicity Formula}
In this section we prove the monotonicity behaviour of the density ratio, as in Remark \eqref{monotonicityspecialcase}, but instead of the hypothesis of such we assume $L^p$ boundedness for $H$, i.e. under the hypothesis \hyperlink{$(AC)$}{$(AC)$}.\\

First to all, from now on we assume  $(M,g)$ to be an $n$-dimensional complete Riemannian manifold satisfying \hyperlink{$(BG)$}{$(BG)$}, and let
\[
c^*=c^*(b,r_0)=\frac{c(b,r_0)-1}{r_0},
\]
where $c(b,s)$ is as in Remark \eqref{comentsHSlemma}.\\

Notice that, if $V\in\V_k(M^n)$ satisfy \hyperlink{$(AC)$}{$(AC)$}, and considering the special case on the Remark \eqref{monotonicityspecialcase} we have in distributional sense that,

\begin{equation}\label{monotonicityinequality1}
\frac d{ds}\lp\frac{\|V\|\lp B_g(\xi,s)\rp}{s^k}\rp\ge\frac d{ds}J(s)+\frac1{s^{k+1}}\int_{B_g(\xi,s)}\lp kc^*s+\Ll H_g,u\nabla u\Rl_g\rp\dv,
\end{equation}
where $J(s)$ is as in Theorem \ref{wightedmonotonicitytheorem}.
The equation above is the Intrinsic Riemannian version of the \textit{fundamental monotonicity identity} of \cite{Simon}, and from now on we refer to this as \textit{fundamental monotonicity inequality} , also notice that, since 
\[
J(s):=\int_{G_k(B_g(\xi,s))}\frac{\lmo\nabla^{T^{\bot}}u\rmo_g^2}{r^k_{\xi}}dV\quad \text{and}\quad \|V\|(B_g(\xi,s)),
\]
are increasing in $s$, the inequality \eqref{monotonicityinequality1} holds also in the classical sense. Furthermore, from the analysis of this inequality, naturally follows a monotone behaviour of the \textit{density ratio}.
\begin{theorem}\label{L1monotonicity}
Given $\xi\in M$, let $0<\alpha\leq 1$ and $\Lambda$ be a positive constant. Let $V\in\V_k(M)$ satisfying \hyperlink{$(AC)$}{$(AC)$} and for all $0<s<r_0$
\[
\frac{1}{\alpha}\int_{B_g(\xi,s)}|H_g(x)|_gd\|V\|(x)\leq \Lambda\left(\frac{s}{r_0}\right)^{\alpha-1}\|V\|\left(B_g(\xi,s)\right).
\] 
Then the function
\[
f(s):=e^{\lambda(s)}\frac{\|V\|\left(B_g(\xi,s)\right)}{s^k},
\]
 is a non-decreasing function, where
 \[
 \lambda(s):=\lp\Lambda\lp\frac s{r_0}\rp^{\alpha-1}-c^*k\rp s,
 \]
 and in fact, for $0<\sigma<\rho\leq r_0$,

\begin{equation}\label{Eq:4.19}
e^{\lambda({\rho})}\frac{\|V\|\left(B_g(\xi,\rho)\right)}{\rho^k} - e^{\lambda({\sigma})}\frac{\|V\|\left(B_g(\xi,\sigma)\right)}{\sigma^k}\geq
J(\rho)-J(\sigma).
\end{equation}

\end{theorem}
\begin{proof}
Notice that it is enough to prove \eqref{Eq:4.19} to guarantee that $f$ is non-decreasing.\\
To prove \eqref{Eq:4.19}, we multiply the \textit{fundamental monotonicity inequality} \eqref{monotonicityinequality1} by  $e^{\lambda(s)}\ge 1$, we have
\begin{align*}
e^{\lambda(s)}\frac d{ds}\lp\frac{\|V\|\lp B_g(\xi,s)\rp}{s^k}\rp
&\ge e^{\lambda(s)}\frac d{ds}J(s)+\frac{e^{\lambda(s)}}{s^{k+1}}\int_{B_g(\xi,s)}\lp kc^*s+\Ll H_g,u\nabla u\Rl_g\rp\dv\\
&\ge \frac d{ds}J(s)-\frac{e^{\lambda(s)}}{s^{k+1}}\int_{B_g(\xi,s)}\lp\abs{H_g}_g\abs{u\nabla u}_g-kc^*s\rp\dv\\
&\ge \frac d{ds}J(s)-\frac{e^{\lambda(s)}}{s^k}\int_{B_g(\xi,s)}\lp\abs{H_g}_g\frac{\abs{u\nabla u}_g}{s}-kc^*\rp\dv\\
&\ge \frac d{ds}J(s)-\frac{e^{\lambda(s)}}{s^k}\int_{B_g(\xi,s)}\lp\abs{H_g}_g-kc^*\rp\dv.
\end{align*}

On the other hand, by hypothesis
\[
-\int_{B_g(\xi,s)}\lp\abs{H}_g-kc^*\rp\dv \ge -\lp\alpha\Lambda\lp \frac{s}{r_0}\rp^{\alpha-1}-kc^*\rp\|V\|(B_g(\xi,s)),
\]
therefore
\begin{equation*}
e^{\lambda(s)}\frac d{ds}\lp\frac{\|V\|\lp B_g(\xi,s)\rp}{s^k}\rp+\lp\alpha\Lambda\lp \frac{s}{r_0}\rp^{\alpha-1}-kc^*\rp e^{\lambda(s)}\frac{\|V\|(B_g(\xi,s))}{s^k}\ge \frac d{ds}J(s).
\end{equation*}
Finally, since
\[
\frac d{ds}\lambda(s)=\alpha\Lambda\lp \frac{s}{r_0}\rp^{\alpha-1}-kc^*,
\]
hence
\[
\frac d{ds}\lp e^{\lambda(s)}\frac{\|V\|(B_g(\xi,s))}{s^k}\rp\ge J(s)\ge0,
\]
and \eqref{Eq:4.19} follows readily by integrating over the interval $[\sigma,\rho]$. 
\end{proof}
\begin{theorem}\label{Thm:LpMonotonicityFormula}
Given $\xi\in M$, let $p>k\ge1$ and $\Gamma$ be a possitive constant. Let $V\in\V_k(M)$ satisfying \hyperlink{$(AC)$}{$(AC)$} such that
\begin{equation*}\label{Eq:LpMonotonicityFormulaStatement}
\left(\int_{B_g(\xi,r_0)}|H_g|_g^pd\|V\|\right)^{1/p}\leq\Gamma.
\end{equation*}
Then there exists a positive constant $c^*_2=c_2^*(p, k, b, r_0)$ such that
\small
\begin{equation}\label{Eq:LpMonotonicityInequality}
\left(\frac{\left(\|V\|(B_g(\xi,\sigma))\right)}{\sigma^k}\right)^{\frac1p}-\left(\frac{\left(\|V\|(B_g(\xi,\rho))\right)}{\rho^k}\right)^{\frac1p}  \leq  \frac{\Gamma+c_2^*}{p-k}\left(\rho^{1-k/p}-\sigma^{1-k/p}\right),
\end{equation}
\normalsize
for every $0<\sigma<\rho\leq r_0$. \\
Furthermore, the functions $f:]0, R[\to]0,+\infty[$ defined as 
$$f(s):=\left(\frac{\left(\|V\|(B_g(\xi,s))\right)}{s^k}\right)^{\frac1p}+\frac{\Gamma+c_1^*}{p-k}s^{1-k/p},$$
 is monotone non-decreasing. 
\end{theorem}
\begin{proof}
From the H\"older inequality in the \textit{fundamental monotonicity inequality} \eqref{monotonicityinequality1} we have for all $0<s<r_0$ that

\begin{align*}
\frac d{ds}\lp\frac{\|V\|(B_g(\xi,s))}{s^k}\rp
&\ge \frac d{ds}J(s)+\frac1{s^{k+1}}\int_{B_g(\xi,s)}\lp kc^*+\Ll H_g,u\nabla u\Rl_g\rp\dv\\
&\ge \frac d{ds}J(s)-\frac1{s^k}\int_{B_g(\xi,s)}\lp\abs{H_g}_g-kc^*\rp\dv\\
&\ge \frac d{ds}J(s)-\frac1{s^k}\int_{B_g(\xi,s)}\abs{H_g}_g\dv+kc^*\frac{\|V\|(B_g(\xi,s))}{s^k}\\
&\ge \frac d{ds}J(s)-\frac1{s^k}\lp \|V\|(B_g(\xi,s)) \rp^{\frac1{p'}}\|H_g\|_{L^p(B_g(\xi,s))}+kc^*\frac{\|V\|(B_g(\xi,s))}{s^k},
\end{align*}

where $p'$ is the conjugate exponent to $p$. Now, by hypothesis we know that 

\[
\|H_g\|_{L^p(\|V\|,B_g(\xi,s))}\le\|H_g\|_{L^p(\|V\|,B_g(\xi,r_0))}\le\Gamma,
\]
and, we also have that $J(s)\ge0$ for all $0<s<r_0$, then
\[
\frac d{ds}\lp\frac{\|V\|(B_g(\xi,s))}{s^k}\rp\ge-\frac{\lp\|V\|(B_g(\xi,s))\rp^{\frac1{p'}}}{s^k}\lp\Gamma-kc^*\lp\|V\|(B_g(\xi,s)).\rp^{\frac1{p}}\rp
\]
Now, by Remarks \ref{behaivourc} and \ref{comentsHSlemma}, $c^*\le0$, then for all $0<s<r_0$
\[
\Gamma-kc^*\lp\|V\|(B_g(\xi,s))\rp^{\frac1{p}}\le\Gamma-kc^*\lp\|V\|(B_g(\xi,r_0))\rp^{\frac1{p}}.
\]
Hence
\begin{equation*}
\frac d{ds}\lp\frac{\|V\|(B_g(\xi,s))}{s^k}\rp\ge-\frac{\lp\|V\|(B_g(\xi,s))\rp^{\frac1{p'}}}{s^k}\lp\Gamma+c^*_2\rp,
\end{equation*}
where $c_2^*=c_2^*(k,p,b,r_0)$ is a positive constant such that
\[
-kc^*\lp\|V\|(B_g(\xi,r_0))\rp^{\frac1{p}}\le c_2^*,
\]
such constant always exists because $\|V\|$ is a Radon measure.\\

On the other hand, since
\[
\frac d{ds}\lp\frac{\|V\|(B_g(\xi,s))}{s^k}\rp^{\frac1p}=\frac1p\lp\frac{\|V\|(B_g(\xi,s))}{s^k}\rp^{\frac1{p'}}\frac d{ds}\lp\frac{\|V\|(B_g(\xi,s))}{s^k}\rp,
\]
we have that
\[
\frac d{ds}\lp\frac{\|V\|(B_g(\xi,s))}{s^k}\rp^{\frac1p}\ge\frac{\Gamma+c^*_2}{ps^{\frac kp}},
\]
and the result follows integrating the inequality above over $[\sigma,\rho]\subset]0,r_0[$.
\end{proof}
\begin{corollary}
If $H\in L_{loc}^p(M,TM,\mathcal{H}_g^k)$ for some $p>k$, then the density
\[
\Theta^k(\|V\|,x):=\lim_{\rho\downarrow 0}\frac{\|V\|(B_g(x,\rho))}{\omega_k\rho^k},
\]
does exists at every $x\in B_g(\xi,r_0)$. Furthermore, $\Theta^k(\|V\|,\cdot)$ is an upper-semi-continuous function in $M^n$, i.e.
\[
\Theta^k(\|V\|,x)\geq\limsup_{y\to x}\Theta^k(\|V\|,x)\quad\forall x\in B_g(\xi,r_0).
\]
\end{corollary}
\begin{proof}
Let
\[
f(s):=\left(\frac{\|V\|(B_g(\xi,s))}{s^k}\right)^{1/p}+\frac{\Gamma+c_1^*}{p-k}s^{1-k/p}.
\]
By the Theorem \ref{Thm:LpMonotonicityFormula}, for $0<\sigma<\rho\leq r_0$
\begin{align*}
f(\rho)-f(\sigma)=&-\left(\left(\frac{\|V\|(B_g(\xi,\sigma))}{\sigma^k}\right)^{1/p}-\left(\frac{\|V\|(B_g(\xi,\rho))}{\rho^k}\right)^{1/p}\right)\\
&+\frac{\Gamma+ c_1^*}{p-k}\left(\rho^{1-1/p}-\sigma^{1-1/p}\right)\geq 0.
\end{align*}
Then $f$ is a non-decreasing function, therefore, the limit $f(s)\to l$ when $s\downarrow 0^+$, exists with $l\in[0,+\infty[$. Hence the limit
\[
\lim_{s\downarrow 0}\frac{\|V\|(B_g(\xi,s))}{s^k},
\] 
also exists, furthermore, is equal to
\[
\lim_{s\downarrow 0}\frac{\|V\|(\overline{B_g(\xi,s)})}{s^k}.
\] 
On the other hand, let $\varepsilon>0$ fixed, such that $0<\sigma<\rho$, $B_g(x,\rho+\varepsilon)\subset  B_g(\xi,r_0)$ and $d_g(x,y)<\varepsilon$, then, we also deduce from Theorem \ref{Thm:LpMonotonicityFormula} and the monotonicity of $\|V\|$ that
\begin{align*}
\left(\frac{\|V\|(B_g(y,\sigma))}{\sigma^k}\right)^{1/p}&\leq\left(\frac{\|V\|(B_g(y,\rho))}{\rho^k}\right)^{1/p}+\frac{\Gamma+c_1^*}{p-k}\rho^{1-k/p}\\
&\leq\left(\frac{\|V\|(B_g(x,\rho+\varepsilon))}{\rho^k}\right)^{1/p}+\frac{\Gamma+c_1^*}{p-k}\rho^{1-k/p}.
\end{align*}
Letting $\sigma\downarrow 0$, we thus have
\[
\left(\Theta^k(\|V\|,y)\right)^{1/p}\leq\left(\frac{\|V\|(B_g(x,\rho+\varepsilon))}{\omega_k(\rho+\varepsilon)^{k/p}}\right)^{1/p}+\frac{\Gamma+c_1^*}{\omega_k^{1/p}(p-k)}\rho^{1-k/p}.
\]
Now, let $\delta>0$ be given and choose $\varepsilon\ll\rho<\delta$ so that
\[
\left(\frac{\|V\|(B_g(x,\rho+\varepsilon))}{\omega_k(\rho+\varepsilon)^{k/p}}\right)<\left(\Theta^k(\|V\|,x)\right)^{1/p}+\delta.
\] 
The inequality above gives
\[
\left(\Theta^k(\|V\|,y)\right)^{1/p}\leq\left(\Theta^k(\|V\|,x)\right)^{1/p}+\frac{\Gamma+c_1^*}{\omega_k^{1/p}(p-k)}\rho^{1-k/p},
\]
provided $d_g(y,x)<\varepsilon$ and  the desired upper-semi-continuity follows straightforward.
\end{proof}
\begin{proposition}\label{prop1}
Let $V\in\V_k(M)$, $\xi\in M$ fixed, and assume that
\[
\|V\|(B_g(\xi,\sigma))\leq\beta\sigma^n\quad for\ 0<\sigma<r_0,
\]
then
\[
\int_{B_g(\xi,\rho)}d_g(x,\xi)^{\alpha-k}d\|V\|(x)\leq \frac{k\beta\rho^{\alpha}}{\alpha},
\]
for any $\rho\in]0,r_0[$ and $0<\alpha<k$.
\end{proposition}
The proof of this is based in the following technical lemma.
\begin{lemma}\label{riem4.3.5}
Let $(X,\mu)$ a measure space, $\gamma>0$, $f\in L^1(\mu)$,  $f\geq 0$, then
\[
\int_0^{\infty}t^{\gamma-1}\mu(E_t)dt=\frac{1}{\gamma}\int_{E_0}f^{\gamma}(x)d\mu(x),
\]
where $E_t:=\{x:f(x)> t\}$. More generally
\[
\int_{t_0}^{\infty}t^{\gamma-1}\mu(E_t)dt=\frac{1}{\gamma}\int_{E_{t_0}}(f^{\gamma}(x)-t_0^{\gamma})d\mu(x),
\]
for each $t_0\geq 0$.
\end{lemma}
This is a clasical Lemma from measure theory, and its proof can be found in \cite{Mattila} pg. 15.
\begin{proof}[Proof of Proposition \ref{prop1}]
Let in the previous Lemma, $f(x)=|x-\xi|^{-1}$, $\gamma=k-\alpha$ and $t_0=1/\rho$, then 
\begin{equation}\label{lemma}
\int_{E_{1/ \rho}}\left(|x-\xi|^{\alpha-k}_g-\rho^{\alpha-k}\right)d\|V\|(x)=(\alpha-k)\int_{1/\rho}^{\infty}t^{(k-\alpha)-1}\|V\|(E_t)d\|V\| (x).
\end{equation}
Notice that
$$E_t:=\{x:f(x)>t\}=\left\{x:\frac{1}{|x-\xi|_g}>t\right\}=\{x:|x-\xi|_g<1/t\}:=B_g(\xi,1/t).$$
On the other hand, by hypothesis, we have
\[
\int_{1/\rho}^{\infty}t^{(k-\alpha)-1}\|V\|(B_g(\xi,1/t))d\mu(x)\leq\int_{1/\rho}^{\infty}\frac{t^{(k-\alpha)-1}\beta}{t^k}d\|V\|(x)=\frac{\beta\rho^{\alpha}}{\alpha},
\]
for any $0<1/t<r_0$.
Putting together the information above in \eqref{lemma}, we have
\small
\begin{align*}
\int_{B_g(\xi,\rho)}\left(|x-\xi|^{\alpha-k}_g-\rho^{\alpha-k}\right)d\|V\|&=\int_{B_g(\xi,\rho)}\left(|x-\xi|^{\alpha-k}_g\right)d\|V\|-\int_{B_g(\xi,\rho)}\rho^{\alpha-k}d\|V\|\\
&=\int_{B_g(\xi,\rho)}\left(|x-\xi|^{\alpha-k}_g\right)d\|V\|-\rho^{\alpha-k}\|V\|(B_g(\xi,\rho))\\
&\leq\frac{(k-\alpha)\beta\rho^{\alpha}}{\alpha},\\
\intertext{then,}
\int_{B_g(\xi,\rho)}\left(|x-\xi|^{\alpha-k}_g\right)d\|V\|&\leq\frac{(k-\alpha)\beta\rho^{\alpha}}{\alpha}+\rho^{\alpha-k}\|V\|(B_g(\xi,\rho))\\
&\leq\frac{(k-\alpha)\beta\rho^{\alpha}}{\alpha}+\beta\rho^{k}\\
&=\frac{k\beta\rho^{\alpha}}{\alpha},
\end{align*}
\normalsize
provided $\rho\in]0,r_0[$.
\end{proof}

\section{Poincar\'e and Sobolev Type Inequalities for Intrinsic Varifolds }
In this final section, a Poincar\'e and Sobolev-type inequalities are proved for non-negative functions defined on $S\subset M^n$, $ \Ha_g^k$-rectifiable sets as a consequence of the \textit{Fundamental Weighted Monotonicity Inequality} \ref{fundamentalinequality}  in the particular case in which $V:=\rv{S}{\theta}$, with $\theta\in L^1_{loc}(\Ha^k_g\llcorner S)$ and mean curvature vector $H_g\in L^1_{loc}(\|V\|)$.\\
Therefore along this section we are under the assumptions of Theorem \ref{wightedmonotonicitytheorem}, i.e. 
\begin{itemize}
\item[\hyperlink{$(BG)$}{$(BG)$}:]Let $(M^n,g)$ be a complete Riemannian Manifold with Levi-Civita connection $\nabla$. We say that $(M^n,g)$ satisfies \hyperlink{$(BG)$}{$(BG)$}, if $Sec_g\leq b$ for some constant $b\in\R$, in case $b>0$ we assume furthermore $br_0<\pi$.
 \item[\hyperlink{$(AC)$}{$(AC)$}:]We say that $V\in\V_k(M)$ satisfies the \hyperlink{$(AC)$}{$(AC)$} condition, whenever
\begin{equation*}
\delta V(X)\le C(V)\lp\int_{B_g(\xi,r_0)}|X|^{\frac{p}{p-1}}d\|V\|\rp^{\frac{p-1}{p}},\;\forall X\in\X^1_c(B_g(\xi,r_0)),
\end{equation*}
for some $p>k$.
\end{itemize}
Furthermore, we need to explain some notation: in what follows $\xi\in M$ is a fixed point and $r_0<inj_{\xi}(M,g)$ such that $r_0\cot_b(r_0)>0$, according to Definition \ref{Def:GeometricConditions}.

\subsection{Poincar\'e type inequality}
The Poincar\'e type inequality which we proof is a consequence of the following lemma, which is a direct application of the formula \eqref{fundamentalinequality}.
\begin{lemma}\label{lemmacondition1ps}
Assume $(M^n,g)$ satisfying \hyperlink{$(BG)$}{$(BG)$} and $V\in\V_k(M^n)$  satisfying \hyperlink{$(AC)$}{$(AC)$}. Let $\Lambda>0$ such that $|H_g|_g\le\Lambda$, and $h\in C^1(M)$, then for $\|V\|$-a.e. $\xi\in\spt \|V\|$, and for all $0<\rho<r_0$
\begin{equation}\label{condition1ps}
h(\xi)\le e^{\lp\Lambda+c^*k\rp\rho}\lp\frac1{\omega_k}\rho^k\int_{B_g(\xi,\rho)}hd\|V\|+\int_{B_g(\xi,\rho)}\frac{\lmo\nabla_g^Sh\rmo_g}{r^{k-1}_{\xi}}d\|V\|\rp,
\end{equation}
where $c^*$ is as in Remark \ref{behaivourc}.
\end{lemma}
\begin{proof}
By Hypothesis \hyperlink{$(BG)$}{$(BG)$} and \hyperlink{$(AC)$}{$(AC)$}  we are in the conditions of Theorem \ref{wightedmonotonicitytheorem}, then, for $\xi\in\spt \|V\|$ and for all $0<s<r_0$ we have

\begin{align*}
\frac{d}{ds}\lp\frac1{s^k}\int_{B_g(\xi,s)}h\g{u} d\|V\|\rp
&\ge\int_{B_g(\xi,s)}\frac h{s^{k}}\frac{\partial}{\partial s}\lp\g{u}\rp\lmo \nabla^{\perp}_gu\rmo^2_g\dv\\
&+\frac1{s^k}\int_{\B{\xi}{s}}\g{u}\Ll\nabla_g^S h,u\nabla_g u\Rl_g\dv\\
&+\frac{kc^*}{s^k}\int_{\B{\xi}{s}}h\g{u}\dv\\
&+\frac1{s^{k+1}}\int_{\B{\xi}{s}}\g{u}\Ll hH_g,u\nabla_g u\Rl_g\dv,
\end{align*}

where $\gamma_{\varepsilon}\in C^1(\R)$ is, such that, given $0<\varepsilon<1$,
\[
\gamma_{\varepsilon}(s):=
\begin{cases}
1,\quad if\quad s<\varepsilon\\
0,\quad if\quad s>1,
\end{cases}
\]
and
\[
\gamma'_{\varepsilon}(s)<0\quad if\quad \varepsilon<s<1.
\]
Then

\begin{align*}
\frac d{ds}\lp\frac1{s^k}\int_{\B{\xi}{s}}h\g{u}\dv\rp&\ge\frac1{s^{k+1}}\int_{\B{\xi}{s}}\g{u}\Ll H_gh+\nabla_g^Sh,u\nabla_g u\Rl_g\dv\\
&+\frac{c^*k}{s^k}\int_{\B{\xi}{s}}h\g{u}\dv\\
&\ge-\frac{1}{s^{k+1}}\g{u}\lmo hH_g+\nabla_g^Sh\rmo_gu\dv\\
&+\frac{c^*k}{s^k}\int_{\B{\xi}{s}}h\g{u}\dv,
\end{align*}

therefore

\begin{align*}
\frac d{ds}\lp\frac1{s^k}\int_{\B{\xi}{s}}\g{u}h\dv\rp&\ge-\frac1{s^{k+1}}\int_{\B{\xi}{s}}\g{u}u\lmo\nabla_g^Sh\rmo_g\dv\\
&-\frac1{s^k}\int_{\B{\xi}{s}}\g{u}h\lp\lmo H_g\rmo_g-c^*k\rp\dv.
\end{align*}

Since, by hypothesis we have that $\lmo H_g\rmo_g\le\Lambda$, we have
\begin{align*}
\frac d{ds}\lp\frac1{s^k}\int_{\B{\xi}{s}}\gamma_{\varepsilon}\lp\frac us\rp h d\|V\|\rp&\ge-\frac1{s^{k+1}}\int_{\B{\xi}{s}}\gamma_{\varepsilon}\lp\frac us\rp u\lmo\nabla_g^S h\rmo_gd\|V\|\\
&-\frac{\Lambda+c^*k}{s^k}\int_{\B{\xi}{s}}\gamma_{\varepsilon}\lp \frac us\rp hd\|V\|.
\end{align*}
Now, let $\{\varepsilon_j\}_{j\in\N}$ a sequence of real numbers such that $\varepsilon\uparrow 1$ as $j\to\infty$, and $\{\gamma_{\varepsilon_j}\}_{j\in\N}\subset C^1(\R)$ such that
\[
\gamma_{\varepsilon_j}(s):=
\begin{cases}
1,\quad if\quad s<\varepsilon_j\\
0,\quad if\quad s>1,
\end{cases}
\]
and
\[
\gamma'_{\varepsilon_j}(s)<0,\quad if\quad \varepsilon_j<s<1,
\]
and $\varepsilon_j\to\chi_{]-\infty,1[}$ pointwise from below. Then, reasoning as above we have  (as in the proof of Theorem \ref{wightedmonotonicitytheorem}) that for all $j\in\N$, and for all $0<s<r_0$

\begin{align*}
\frac d{ds}\lp\frac1{s^k}\int_{\B{\xi}{s}}\gamma_{\varepsilon_j}\lp\frac us\rp h d\|V\|\rp&\ge-\frac1{s^{k+1}}\int_{\B{\xi}{s}}\gamma_{\varepsilon_j}\lp\frac us\rp u\lmo\nabla_g^S h\rmo_gd\|V\|\\
&-\frac{\Lambda+c^*k}{s^k}\int_{\B{\xi}{s}}\gamma_{\varepsilon_j}\lp \frac us\rp hd\|V\|.
\end{align*}

Multiply the above inequality by $e^{\tilde{\lambda}(s)}$ and letting $j\to\infty$ we have in distributional sense  in $s$,
\small
\begin{equation}
\begin{split}
e^{\tilde{\lambda}(s)}\frac d{ds}\lp\frac1{s^k}\int_{B_g(\xi,s)} h d\|V\|\rp
&+e^{\tilde{\lambda}(s)}\frac{\Lambda+c^*k}{s^k}\int_{B_g(\xi,s)} hd\|V\|\\
&\ge -e^{\tilde{\lambda}(s)}\frac1{s^{k+1}}\int_{B_g(\xi,s)}u\lmo\nabla_g^S h\rmo_gd\|V\|,
\end{split}
\end{equation}
\normalsize
where
\begin{equation}\label{tildelambda(s)}
\tilde{\lambda}(s):=\lp\Lambda+c^*k\rp s,
\end{equation}
therefore, since $\tilde{\lambda}'(s)=\Lambda+c^*k$, we have 

\begin{equation*}
\frac d{ds}\lp\frac{e^{\tilde{\lambda}(s)}}{s^k}\int_{B_g(\xi,s)} hd\|V\|\rp\ge-\frac{e^{\tilde{\lambda} (s)}}{s^{k+1}}\int_{b_g(\xi,s)} u\lmo\nabla^S_g h\rmo_gd\|V\|.
\end{equation*}

Hence, integrating over $[\sigma,\rho]\subset]0,r_0[$,
\small
\begin{align*}
\frac{e^{\tilde{\lambda}(\rho) }}{\rho^k}\int_{B_g(\xi,\rho)}hd\|V\|&-\frac{e^{\tilde{\lambda}(\sigma)}}{\sigma^k}\int_{B_g(\xi,\sigma)}hd\|V\|=\int_{\sigma}^{\rho}\frac d{ds}\lp\frac{e^{\tilde{\lambda}(s)}}{s^k}\int_{B_g(\xi,s)} hd\|V\|\rp\\
&\ge-\int_{\sigma}^{\rho}\frac{e^{\tilde{\lambda} (s)}}{s^{k+1}}\int_{B_g(\xi,s)}u\lmo\nabla^S_g h\rmo_gd\|V\|\\
&\ge-e^{\tilde{\lambda} (\rho)}\int_{\sigma}^{\rho}\frac{1}{s^{k+1}}\lp\int_{B_g(\xi,s)}u\lmo\nabla^S_g h\rmo_gd\|V\|\rp ds,
\end{align*}
\normalsize
then, rearranging terms, 
\small
\begin{multline}
\frac{1}{\omega_k\sigma_k}\int_{B_g(\xi,\sigma)}h d\|V\|\\
\le e^{\tilde{\lambda} (\rho)}\lp\frac{1}{\omega_k\rho_k}\int_{B_g(\xi,\rho)}h d\|V\|+\frac{1}{\omega_k}\int_{\sigma}^{\rho}\frac{1}{s^{k+1}}\lp\int_{B_g(\xi,s)}u\lmo\nabla^S_g h\rmo_gd\|V\|\rp ds\rp.
\end{multline}
\normalsize
Now, letting $\tau=1/s$ in the inequality above, we have that
\[
\int_{\sigma}^{\rho}\lp\frac{1}{s}\rp^{k+1}\lp\int_{B_g(\xi,s)}u\lmo\nabla^S_g h\rmo_gd\|V\|\rp ds=-\int_{\frac1{\sigma}}^{\frac1{\rho}}\tau^{k-1}\int_{B_g\lp\xi,\frac1{\tau}\rp}u\lmo\nabla^S_g h\rmo_gd\|V\|,
\]
then, if we let for $A\subset M$
\[
\nu(A):=\int_Au\lmo\nabla^S_g h\rmo_gd\|V\|,
\]
since $u$ is an increasing function, $\nu$ defines a measure on $(M,g)$, so by Lemma \ref{riem4.3.5}
\begin{align*}
\int_{\sigma}^{\rho}\frac{1}{s^{k+1}}\lp\int_{B_g(\xi,s)}u\lmo\nabla^S_g h\rmo_gd\|V\|\rp ds&=-
\int_{\frac1{\sigma}}^{\frac1{\rho}}\tau^{k-1}\nu\lp B_g\lp\xi,\frac1{\tau}\rp\rp d\tau\\
&=\int_{\frac1{\rho}}^{\frac1{\sigma}}\tau^{k-1}\nu\lp\left\{x\in M: \frac1{u(x)}>\tau\right\}\rp d\tau\\
&=\frac1k\int_{\left\{\frac1{\rho}<\frac1{u}<\frac1{\sigma}\right\}}\frac1{u(x)^k}-\lp\frac1{\sigma^k}-\frac1{\rho^k}\rp d\nu\\
&=\int_{B_g(\xi,\rho)\setminus B_g(\xi,\sigma)}\frac1{u^k}-\lp\frac1{\sigma^k}-\frac1{\rho^k}\rp d\nu\\
&\le\frac1k\int_{B_g(\xi,\rho)}\frac1{ u^k}d\nu\\
&=\frac1k\int_{B_g(\xi,\rho)}\frac{u\lmo\nabla^S_g h\rmo_g}{u^k}d\|V\|\\
&=\frac1k\int_{B_g(\xi,\rho)}\frac{\lmo\nabla^S_g h\rmo_g}{u^{k-1}}d\|V\|.
\end{align*}
Then
\small
\[
\frac1{\omega_k\sigma^k}\int_{B_g(\xi,\sigma)}h d\|V\|\le e^{\tilde{\lambda}(\rho)}\lp\frac1{\omega_k\rho^k}\int_{B_g(\xi,\rho)}h d\|V\|+\frac1{k\omega_k}\int_{B_g(\xi,\rho)}\frac{\lmo\nabla^S_g h\rmo_g}{r_{\xi}^{k-1}}d\|V\|\rp.
\]

\normalsize

Finally, let $\sigma\downarrow 0$, and the result follows, since, $\xi\in\spt \|V\|$ implies $\Theta^k(\|V\|,\xi)=1$ for $\|V\|$-a.e. $\xi$, therefore $\xi$ is a Lebesgue point $\|V\|$-a.e.
\end{proof}
\begin{theorem}[Poincar\'e-Type Inequality for Intrinsic Varifolds]
Let $(M^n,g)$ be a complete Riemannian manifold satisfying \hyperlink{$(BG)$}{$(BG)$}, let $V:=\rv{S}{\theta}$ a rectifiable varifold satisfying \hyperlink{$(AC)$}{$(AC)$}. Suppose: $h\in C^1(M)$, $h\ge 0$, $B_g(\xi,2\rho)\subset B_g(\xi, r_0)$ for $\xi\in S$ fixed, $\lmo H_g\rmo_g\le\Lambda$ for some $\Lambda>0$, $\theta>1$ $\|V\|$-a.e. in $B_g(\xi,r_0)$ and for some $0<\alpha<1$
\begin{equation}\label{poincarehypothesis}
\begin{split}
\|V\|\lp\left\{ x\in B_g\lp\xi,\rho\rp:h(x) > 0\right\}\rp&\le\omega_k\lp1-\alpha\rp\rho^k\quad and,\\ e^{\lp\Lambda+kc^*\rp\rho}&\le 1+\alpha.
\end{split}
\end{equation} 

Suppose also that, for some constant $\Gamma>0$

\begin{equation}
\|V\|\lp B_g(\xi,2\rho)\rp\le\Gamma\rho^k.
\end{equation}

Then there are constants $\beta:=\beta\lp k,\alpha,r_0,b\rp\in]0,\frac12[$ and $C:=C\lp k,\alpha,r_0,b\rp>0$ such that
\[
\int_{\B{\xi}{2\rho}}h\dv\le C\rho\int_{\B{\xi}{\rho}}\lmo \nabla^S_g h\rmo_g\dv.
\]

\end{theorem}
\begin{proof}
Take $\beta\in]0,1/2[$ an arbitrary parameter, to be specified  later and $\tilde{\lambda}(s)$ as in \eqref{tildelambda(s)}.  Applying the previous Lemma with
\[
\eta\in\B{\xi}{\beta\rho}\cap\spt\|V\|,
\]
in place of $\xi$, thence, for all $0<\rho<r_0$
\small
\[
h(\eta)\le e^{\tilde{\lambda}((1-\beta)\rho)}\lp\frac1{\omega_k(1-\beta)^k\rho^k}\int_{B_g(\eta,(1-\beta)\rho)}h d\|V\|+\frac1{k\omega_k}\int_{B_g(\eta,(1-\beta)\rho)}\frac{\lmo\nabla^S_g h\rmo_g}{r_{\eta}^{k-1}}d\|V\|\rp.
\]
Since $\rho/2<(1-\beta)\rho<\rho$, we have $\B{\eta}{(1-\beta)\rho}\subset\B{\xi}{\rho}$ and
\begin{equation}\label{eq5.1}
h(\eta)\le e^{\tilde{\lambda}(\rho)}\lp\frac{1}{\omega_k(1-\beta)^k\rho^k}\int_{\B{\xi}{\rho}}h\dv+\frac1{k\omega^k}\int_{\B{\xi}{\rho}}\frac{\lmo\nabla^S_g h\rmo_g}{r_{\eta}^{k-1}}\dv\rp.
\end{equation}
Now, let $t_0\ge0$ fixed, and $\gamma\in C^1(\R)$ a fixed non-decreasing function with
\[
\gamma(t)=0,\ if\ t\le0\quad and\quad \gamma(t)\le1\ if\ t>0,
\]
and apply \eqref{eq5.1} with $f(x):=\gamma\lp h(x)-t_0\rp$ in place of $h$, then
\small
\[
f(\eta)\le e^{\tilde{\lambda}(\rho)}\lp\frac{1}{\omega_k(1-\beta)^k\rho^k}\int_{\B{\xi}{\rho}}f\dv+\frac1{k\omega^k}\int_{\B{\xi}{\rho}}\frac{\gamma'\lp h-t_0\rp\lmo\nabla^S_g h\rmo_g}{r_{\eta}^{k-1}}\dv\rp.
\]
\normalsize
By hypothesis \eqref{poincarehypothesis} we know that:
\begin{align*}
 e^{\tilde{\lambda}(\rho)}&\le1+\alpha\\
 \intertext{and that}
 \|V\|\lp\left\{x\in\B{\xi}{\rho}:h(x)>0\right\}\rp&\le\omega_k(1-\alpha)\rho^k, 
 \end{align*}

hence, 

\small
\begin{align*}
f(\eta)
&\le\lp1+\alpha\rp\lp\frac{1}{\omega_k(1-\beta)^k\rho^k}\int_{\B{\xi}{\rho}}f\dv+\frac1{k\omega^k}\int_{\B{\xi}{\rho}}\frac{\gamma'\lp h-t_0\rp\lmo\nabla^S_g h\rmo_g}{r_{\eta}^{k-1}}\dv\rp\\
&\le\lp1+\alpha\rp\lp\frac{\|V\|\lp\left\{h>0\right\}\cap\B{\xi}{\rho}\rp}{\omega_k(1-\beta)^k\rho^k}+\frac1{k\omega^k}\int_{\B{\xi}{\rho}}\frac{\gamma'\lp h-t_0\rp\lmo\nabla^S_g h\rmo_g}{r_{\eta}^{k-1}}\dv\rp\\
&\le\frac{1-\alpha^2}{(1-\beta)^k}+\frac{1+\alpha}{k\omega_k}\int_{\B{\xi}{\rho}}\frac{\gamma'\lp h-t_0\rp\lmo\nabla^S_g h\rmo_g}{r_{\eta}^{k-1}}\dv.
\end{align*}
\normalsize
Now, take $\beta:=\beta(k,\alpha)\in]0,1/2[$ small enough, such that
\[
\frac{1-\alpha^2}{(1-\beta)^k}\le1-\frac{\alpha^2}2,
\]
then
\begin{equation}\label{eq5.3}
\frac{\alpha^2}2+\lp f(\eta)-1\rp\le\frac{\lp1+\alpha\rp}{k\omega_k}\int_{\B{\xi}{\rho}}\frac{\gamma'\lp h(x)-t_0\rp\lmo\nabla^S_g h(x)\rmo_g}{r_{\eta}^{k-1}}\dv(x).
\end{equation}
Therefore, for any $\eta\in\B{\xi}{\beta\rho}\cap\spt\|V\|$ such that $f(\eta)\ge1$,
\begin{equation}\label{eq5.2}
1\le\frac{2(1+\alpha)}{\alpha^2k\omega_k}\int_{\B{\xi}{\rho}}\frac{\gamma'\lp h(x)-t_0\rp\lmo\nabla^S_g h(x)\rmo_g}{r_{\eta}^{k-1}}\dv(x).
\end{equation}
Let $\varepsilon>0$ given,  and choose $\gamma$ such that, $\gamma(t)\equiv 1$ for $t\ge\varepsilon$, then $f\equiv 1$ for $\eta\in\B{\xi}{\beta\rho}\cap E_{t_0+\varepsilon}$, where $E_{\tau}:=\left\{x\in\spt\|V\|: h(x)>\tau\right\}$, and so \eqref{eq5.3} remains valid. Then integrating in booth sides over $\B{\xi}{\beta\rho}\cap E_{t_0+\varepsilon}$, we get
\small
\begin{align*}
\|V\|\lp\B{\xi}{\beta\rho}\cap E_{t_0+\varepsilon}\rp
&\le\frac{2(1+\alpha)}{\alpha^2k\omega_k}\int_{\B{\xi}{\beta\rho}\cap E_{t_0+\varepsilon}}\lp\int_{\B{\xi}{\rho}}\frac{\gamma'\lp h(x)-t_0\rp\lmo\nabla^S_g h(x)\rmo_g}{r_{\eta}^{k-1}}\dv(x)\rp \dv(\eta)\\
&=\frac{2(1+\alpha)}{\alpha^2k\omega_k}\int_{\B{\xi}{\rho}}\gamma'\lp h(x)-t_0\rp\lmo\nabla^S_g h(x)\rmo_g\lp\int_{\B{\xi}{\beta\rho}\cap E_{t_0+\varepsilon}}\frac{1}{r_{\eta}^{k-1}}\dv(\eta)\rp \dv(x)\\
&\le\frac{2(1+\alpha)}{\alpha^2k\omega_k}\int_{\B{\xi}{\rho}}\gamma'\lp h(x)-t_0\rp\lmo\nabla^S_g h(x)\rmo_g\lp\int_{\B{\xi}{\rho}}\frac{1}{r_{\eta}^{k-1}}\dv(\eta)\rp \dv(x).
\end{align*}
\normalsize
Now, since by hypothesis $\|V\|\lp\B{\xi}{2\rho}\rp\le\Gamma\rho^k$, we have, by Proposition \ref{prop1} for all $0<\rho<r_0$ (notice that $\rho$ has to satisfy $0<\rho<2\rho<r_0$)
\[
\int_{\B{\xi}{\rho}}\frac1{r^{k-1}_{\eta}}\dv(\eta)\le k\Gamma\rho,
\]
then
\[
\|V\|\lp\B{\xi}{\beta\rho}\cap E_{t_0+\varepsilon}\rp\le\lp\frac{2(1+\alpha)\Gamma}{\alpha^2\omega_k}\rp\rho\int_{\B{\xi}{\rho}}\gamma'\lp h(x)-t_0\rp\lmo\nabla^S_g h(x)\rmo_g\dv(x).
\]
Integrating over $[0,\infty[$ in $t_0$,
\small
\begin{align*}
\int_0^{\infty}\|V\|\lp\B{\xi}{\beta\rho}\cap E_{t_0+\varepsilon}\rp dt_0
&\le\lp\frac{2(1+\alpha)\Gamma}{\alpha^2\omega_k}\rp\rho\int_0^{\infty}\lp\int_{\B{\xi}{\rho}}\gamma'\lp h(x)-t_0\rp\lmo\nabla^S_g h(x)\rmo_g\dv(x)\rp dt_0\\
&=\lp\frac{2(1+\alpha)\Gamma}{\alpha^2\omega_k}\rp\rho\int_{0}^{\infty}\lp \int_{\B{\xi}{\rho}}-\frac{\partial}{\partial t_0}\lp\gamma(h(x)-t_0)\rp\lmo\nabla^S_g h(x)\rmo_g\dv(x)\rp dt_0\\
&= \lp\frac{2(1+\alpha)\Gamma}{\alpha^2\omega_k}\rp\rho\int_{\B{\xi}{\rho}}\lmo\nabla^S_g h(x)\rmo_g\lp-\frac{d}{ds}\int_0^{\infty}\gamma(h(x)-t_0)dt_0\rp\dv(x)\\
&\le\lp\frac{2(1+\alpha)\Gamma}{\alpha^2\omega_k}\rp\rho\int_{\B{\xi}{\rho}}\lmo\nabla^S_g h(x)\rmo_g\gamma(h(x))\dv(x)\\
&\le\lp\frac{2(1+\alpha)\Gamma}{\alpha^2\omega_k}\rp\rho\int_{\B{\xi}{\rho}}\lmo\nabla^S_g h(x)\rmo_g\dv(x).
\end{align*}
\normalsize
To calculate the left term, let $\tilde{A}_{\tau}:=\left\{x\in\spt\|V\|:(h-\varepsilon)(x)>\tau\right\}=A_{\tau+\varepsilon}$, then by Lemma \ref{riem4.3.5}
\begin{align*}
\int_0^{\infty}\|V\|\lp\B{\xi}{\beta\rho}\cap E_{t_0+\varepsilon}\rp&=\int_0^{\infty}\|V\|\llcorner\lp\B{\xi}{\beta\rho}\rp\lp\tilde{A}_{\tau}\rp dt_0\\
&=\int_{\tilde{A}_0} \lp h-\varepsilon\rp\dv\llcorner\lp\B{\xi}{\beta\rho}\rp\\
&=\int_{\B{\xi}{\beta\rho}\cap A_{\varepsilon}}\lp h-\varepsilon\rp \dv.
\end{align*}
Sumarizing,
\[
\int_{\B{\xi}{\beta\rho}\cap A_{\varepsilon}}\lp h(x)-\varepsilon\rp \dv(x)\le C\rho\int_{\B{\xi}{\rho}}\lmo\nabla^S_g h(x)\rmo_g\dv(x),
\]
where
\[
C:=C(k,\alpha,\Gamma,r_0):=\frac{2(1+\alpha)\Gamma}{\alpha^2\omega_k}.
\]
Finally, letting $\varepsilon\downarrow 0$, the result follows.
\end{proof}
\subsection{Sobolev type Inequality}
\begin{lemma}\label{lemmacondition2ps}
Assume $(M^n,g)$ satisfying \hyperlink{$(BG)$}{$(BG)$} and $V:=\rv{S}{\theta}$, with $S\subset M$ $\Ha^k_g$-rectifiable, and $\theta\in L^1_{loc}(\|V\|)$, such that \hyperlink{$(AC)$}{$(AC)$} is satisfied. Let $h\in C^1(M)$ non-negative, then for all $\xi\in\spt \|V\|$ and for all $0<\rho<r_0$
\small
\begin{equation}\label{condition2ps}
\frac1{\omega_k\sigma^k}\int_{\B{\xi}{\sigma}}h\dv\le\frac1{\omega_k\rho^k}\int_{\B{\xi}{\rho}}h\dv+\int_{\sigma}^{\rho}\frac1{s^k}\lp\int_{\B{\xi}{s}}\lmo\nabla^S_g h\rmo_g+h\lp\lmo H_g\rmo_g-c^*k\rp\dv\rp ds
\end{equation}
\normalsize
where $c^*$ is as in Remark \ref{behaivourc}.
\end{lemma}
\begin{proof}
By Hypothesis \hyperlink{$(BG)$}{$(BG)$} and \hyperlink{$(AC)$}{$(AC)$}  we are in the conditions of Theorem \ref{wightedmonotonicitytheorem}, then, for $\xi\in\spt \|V\|$ and for all $0<s<r_0$ 

\begin{align*}
\frac{d}{ds}\lp\frac1{s^k}\int_{B_g(\xi,s)}h\g{u} d\|V\|\rp
&\ge\int_{B_g(\xi,s)}\frac h{s^{k}}\frac{\partial}{\partial s}\lp\g{u}\rp\lmo \nabla^{\perp}_gu\rmo^2_g\dv\\
&+\frac1{s^k}\int_{\B{\xi}{s}}\g{u}\Ll\nabla_g^S h,u\nabla_g u\Rl_g\dv\\
&+\frac{kc^*}{s^k}\int_{\B{\xi}{s}}h\g{u}\dv\\
&+\frac1{s^{k+1}}\int_{\B{\xi}{s}}\g{u}\Ll H_gh,u\nabla_g u\Rl_g\dv,
\end{align*}

where $\gamma_{\varepsilon}\in C^1(\R)$ is, such that, given $0<\varepsilon<1$,
\[
\gamma_{\varepsilon}(s):=
\begin{cases}
1,\quad if\quad s<\varepsilon\\
0,\quad if\quad s>1,
\end{cases}
\]
and
\[
\gamma'_{\varepsilon}(s)<0\quad if\quad \varepsilon<s<1.
\]
Then

\begin{align*}
\frac d{ds}\lp\frac1{s^k}\int_{\B{\xi}{s}}h\g{u}\dv\rp
&\ge\frac1{s^{k+1}}\int_{\B{\xi}{s}}\g{u}\Ll H_gh+\nabla^S_g h,u\nabla_g u\Rl_g\dv\\
&+\frac{c^*k}{s^k}\int_{\B{\xi}{s}}h\g{u}\dv\\
&\ge-\frac{1}{s^{k+1}}\int_{\B{\xi}{s}}\g{u}\lmo H_gh+\nabla^S_g h\rmo_gu\dv\\
&+\frac{c^*k}{s^k}\int_{\B{\xi}{s}}h\g{u}\dv,
\end{align*}

therefore
\begin{equation*}
\frac d{ds}\lp\frac1{s^k}\int_{\B{\xi}{s}}\g{u}h\dv\rp\ge-\frac1{s^{k}}\int_{\B{\xi}{s}}\g{u}\lp\lmo\nabla^S_g h\rmo_g+h\lp\lmo H_g\rmo_g-c^*k\rp\rp\dv\\.
\end{equation*}
Now, let $\{\varepsilon_j\}_{j\in\N}$ a sequence of real numbers such that $\varepsilon\uparrow 1$ as $j\to\infty$, and $\{\gamma_{\varepsilon_j}\}_{j\in\N}\subset C^1(\R)$ such that
\[
\gamma_{\varepsilon_j}(s):=
\begin{cases}
1,\quad if\quad s<\varepsilon_j\\
0,\quad if\quad s>1,
\end{cases}
\]
and
\[
\gamma'_{\varepsilon_j}(s)<0\quad if\quad \varepsilon_j<s<1,
\]
and $\varepsilon_j\to\chi_{]-\infty,1[}$ pointwise from below. Then, reasoning as in the proof of Theorem \eqref{wightedmonotonicitytheorem}, for all $j\in\N$, and for all $0<s<r_0$
\small
\begin{equation*}
\frac d{ds}\lp\frac1{s^k}\int_{\B{\xi}{s}}\gamma_{\varepsilon_j}\lp\frac us\rp h\dv\rp\ge-\frac1{s^{k}}\int_{\B{\xi}{s}}\gamma_{\varepsilon_j}\lp\frac us\rp\lp\lmo\nabla^S_g h\rmo_g+h\lp\lmo H_g\rmo_g-c^*k\rp\rp\dv\\.
\end{equation*}
\normalsize
Letting $j\to\infty$ we have in distributional sense  in $s$,
\begin{equation*}
\frac d{ds}\lp\frac1{s^k}\int_{\B{\xi}{s}}h\dv\rp\ge-\frac1{s^{k}}\int_{\B{\xi}{s}}\lp\lmo\nabla^S_g h\rmo_g+h\lp\lmo H_g\rmo_g-c^*k\rp\rp\dv\\.
\end{equation*}
Finally, integrating over $[\sigma,\rho]\subset ]0,r_0[$ we have 
\small
\begin{align*}
\int_{\sigma}^{\rho}\frac d{ds}\lp\frac1{s^k}\int_{\B{\xi}{s}}h\dv\rp&=\frac1{\rho^k}\int_{\B{\xi}{\rho}}h\dv-\frac1{\sigma^k}\int_{\B{\xi}{\sigma}}h\dv\\
&\ge-\int_{\sigma}^{\rho}\lp\frac1{s^{k}}\int_{\B{\xi}{s}}\lp\lmo\nabla^S_g h\rmo_g+h\lp\lmo H_g\rmo_g-c^*k\rp\rp\dv\rp ds.
\end{align*}
\normalsize
Therefore,
\small
\[
\frac1{\omega_k\sigma^k}\int_{\B{\xi}{\sigma}}h\dv\le\frac1{\omega_k\rho^k}\int_{\B{\xi}{\rho}}h\dv+\int_{\sigma}^{\rho}\frac1{s^k}\lp\int_{\B{\xi}{s}}\lmo\nabla^S_g h\rmo_g+h\lp\lmo H_g\rmo_g-c^*k\rp\dv\rp ds.
\]
\normalsize
\end{proof}
\begin{lemma}\label{lemmacomputation}
Suppse $f$ and $g$ are bounded non-decreasing functions on $]0,\infty[$, and
\begin{equation}\label{eq5.4}
1\leq\frac1{\sigma^k}f(\sigma)\le\frac1{\rho^k}f(\rho)+\int_{0}^{\rho}\frac1{s^k}g(s)ds,
\end{equation}
where $0<\sigma<\rho<\infty$. Then, there exists $\rho\in]0,\rho_0[$ such that
\[
f(\rho)\le\frac125^k\rho_0g(\rho),
\]
where
\[
\rho_0=2\lp\lim_{\rho\to\infty}f(\rho)\rp^{\frac1k}.
\]
\end{lemma}
\begin{proof}
We argue by contradiction. Assume that for every $\rho\in]0,\rho_0[$
\[
f(5\rho)>\frac12 5^k\rho_0g(\rho),
\]
then, by \eqref{eq5.4},
\begin{equation}\label{eq5.5}
\begin{split}
1\le\sup_{0<\sigma<\rho_0}\frac1{\sigma^k}f(\sigma)&\le\frac1{\rho_0^k}f(\rho_0)+\int_0^{\rho_0}\frac1{s^k}g(s)ds\\
&<\frac1{\rho_0^k}f(\rho_0)+\frac2{5^k\rho_0}\int_0^{\rho_0}\frac1{s^k}f(5s)ds\\
&=\frac1{\rho_0^k}f(\rho_0)+\frac2{5^k\rho_0}\int_0^{5\rho_0}\frac{5^k}{s^k}f(s)\frac{ds}5\\
&=\frac1{\rho_0^k}f(\rho_0)+\frac2{5\rho_0}\lp\int_0^{\rho_0}\frac1{s^k}f(s)ds+\int_{\rho_0}^{5\rho_0}\frac1{s^k}f(s)ds\rp\\
&\le\frac1{\rho_0^k}f(\rho_0)+\frac25\sup_{0<s<\rho_0}\frac1{s^k}f(s)+\frac2{5\rho_0}\lim_{s\to\infty}f(s)\int_{\rho_0}^{5\rho_0}\frac1{s^k}ds\\
&\le\frac1{\rho_0^k}\lim_{s\to\infty}f(s)+\frac25\sup_{0<\sigma<\rho_0}\frac1{\sigma^k}f(\sigma)+\lp\frac2{5(1-n)\rho_0}\lp\left.\frac1{s^{k-1}}\rmo_{\rho_0}^{5\rho_0}\rp\lim_{s\to\infty}f(s)\rp\\
&\le\frac1{\rho_0^k}\lim_{s\to\infty}f(s)+\frac2{5(n-1)}\frac1{\rho_0^k}\lim_{s\to\infty}f(s)+\frac25\sup_{0<\sigma<\rho_0}\frac1{\sigma^k}f(\sigma)\\
&\le\lp1+\frac25\rp\frac1{\rho_0^k}\lim_{s\to\infty}f(s)+\frac12\sup_{0<\sigma<\rho_0}\frac1{\sigma^k}f(\sigma)\\
&\le\lp1+\frac12\rp\frac1{\rho_0^k}\lim_{s\to\infty}f(s)+\frac12\sup_{0<\sigma<\rho_0}\frac1{\sigma^k}f(\sigma),
\end{split}
\end{equation}
then,
\[
\sup_{0<\sigma<\rho_0}\frac1{\sigma^k}f(\sigma)<3\frac1{\rho_0^k}\lim_{s\to\infty}f(s).
\]
Hence
\[
\sup_{0<\sigma<\rho_0}\frac1{\sigma^k}f(\sigma)<\frac12\lp\sup_{0<\sigma<\rho_0}\frac1{\sigma^k}f(\sigma)+3\frac1{\rho_0^k}\lim_{s\to\infty}f(s)\rp<3\frac1{\rho_0^k}\lim_{s\to\infty}f(s)<4\frac1{\rho_0^k}\lim_{s\to\infty}f(s),
\]
and then, by \eqref{eq5.5} we have,
\[
1\le\sup_{0<\sigma<\rho_0}\frac1{\sigma^k}f(\sigma)<\frac12\lp\sup_{0<\sigma<\rho_0}\frac1{\sigma^k}f(\sigma)+3\frac1{\rho_0^k}\lim_{s\to\infty}f(s)\rp<4\frac1{\rho_0^k}\lim_{s\to\infty}f(s).
\]
Finally, since
\[
\rho_0^k=2^k\lim_{s\to\infty}f(s),
\]
and dividing by $2$, we have
\[
\frac12\le\frac12\sup_{0<\sigma<\rho_0}f(\sigma)<\frac1{2^{k-1}},
\]
which clearly is a contradiction sonce $k\ge2$.
\end{proof}
\begin{theorem}[Sobolev-Type Inequality for Intrinsic Varifolds]
Let $(M^n,g)$ be a complete Riemannian manifold satisfying \hyperlink{$(BG)$}{$(BG)$}, let $V=\rv{S}{\theta}$ a rectifiable varifold satisfying \hyperlink{$(AC)$}{$(AC)$}. Suppose $h\in C^1_0(S)$ non negative, and $\theta\ge1$ $\|V\|$-a.e. in $S$. Then there exists $C:=C(k)>0$ such that
\begin{equation}\label{Eq:MichaelSimonInManifolds}
\lp\int_S h^{\frac k{k-1}}\dv\rp^{\frac{k-1}k}\le C\int_S\lp\lmo\nabla_g^S h\rmo_g+h\lp\lmo H_g\rmo_g-c^*k\rp\rp\dv.
\end{equation}
\end{theorem}
\begin{proof}
Notice that, for Lemma \ref{lemmacondition2ps}, for all $0<\sigma<\rho<r_0$
\[
\frac1{\omega_k\sigma^k}\int_{\B{\xi}{\sigma}}h\dv\le\frac1{\omega_k\rho^k}\int_{\B{\xi}{\rho}}h\dv+\int_{\sigma}^{\rho}\frac1{s^k}\lp\int_{\B{\xi}{s}}\lmo\nabla^S_g h\rmo_g+h\lp\lmo H_g\rmo_g-c^*k\rp\dv\rp ds,
\]
but, since by hypothesis $h\in C_0^1(\B{\xi}{r_0})$ it remains valid for all $0<\sigma<\rho<\infty$, hence we can apply the Lemma \ref{lemmacomputation} with the choices
\begin{align*}
f(\rho):=&\frac1{\omega_k}\int_{\B{\xi}{\rho}}h\dv\\
g(\rho):=&\frac1{\omega_k}\int_{\B{\xi}{\rho}}\lp\lmo\nabla^S_g h\rmo_g+h\lp\lmo H_g\rmo_g-c^*k\rp \rp\dv,
\end{align*}
provided that $\xi\in\spt\|V\|$ and $h(\xi)\ge1$.\\
Thus for $\|V\|$-a.e $\xi\in\left\{x\in\spt \|V\|:h(x)\ge1\right\}$ we have
\[
\rho<2\lp\frac1{\omega_k}\int_{M}h\dv\rp^{\frac1k},
\]
and
\begin{equation}\label{eq5.5}
\int_{\B{\xi}{5\rho}} h\dv\le5^k\lp\frac1{\omega_k}\int_{M}h\dv\rp^{\frac1k}\int_{\B{\xi}{\rho}} \lp\lmo\nabla^S_g h\rmo_g+h\lp\lmo H_g\rmo_g-c^*k\rp\rp\dv.
\end{equation}
Now, since $\left\{ x\in\spt\|V|:h(x)\ge1\right\}\subset\spt\|V\|\cap\spt h$ which is compact, in virtue of  Vitali's five Lemma (cf. \cite{Simon} Theorem 3.3 pg. 11) we can select a collection of balls
\[
\left\{\B{\xi_i}{\rho_i}\right\}_{i\subset\Lambda},
\]
such that, for all $i\in\Lambda$, $\xi_i\subset\left\{x\in\spt\|V\|:h(x)\ge1\right\}$ and, there exists $\Lambda'\subset\Lambda$ such that for every $i\ne j$, $i,j\in\Lambda'$ it holds $\B{\xi_i}{\rho_i}\cap\B{\xi_j}{\rho_j}=\emptyset$ and 
\[
\left\{x\in\spt\|V\|:h(x)\ge1\right\}\subset\bigcup_{j\in\Lambda'}\B{\xi_{i_j}}{5\rho_{i_j}}=\bigcup_{i\in\Lambda}\B{\xi_i}{\rho_i}.
\]
Then, by \eqref{eq5.5}
\small
\begin{equation}\label{eq5.6}
\begin{split}
\int_{\left\{h\ge1\right\}\cap\spt\|V\|}h\dv&
\le\sum_{j\in\Lambda'}\int_{\B{\xi_{i_j}}{5\rho_{i_j}}}h\dv\\
&\le\sum_{j\in\Lambda'}\lp5^k\lp\frac1{\omega_k}\int_Mh\dv\rp^{\frac1k}\int_{\B{\xi_{i_j}}{\rho_{i_j}}}\lp\lmo\nabla^S_g h\rmo_g+h\lp\lmo H\rmo_g-c^*k\rp\rp\dv\rp\\
&\le5^k\lp\frac1{\omega_k}\int_{M}h\dv\rp^{\frac1k}\sum_{j\in\Lambda'}\int_{\B{\xi_{i_j}}{\rho_{i_j}}}\lp\lmo\nabla^S_g h\rmo_g+h\lp\lmo H\rmo_g-c^*k\rp\rp\dv\\
&\le5^k\lp\frac1{\omega_k}\int_{M}h\dv\rp^{\frac1k}\int_M\lp\lmo\nabla^S_g h\rmo_g+h\lp\lmo H\rmo_g-c^*k\rp\rp\dv.
\end{split}
\end{equation}
\normalsize
Let $\gamma_\varepsilon\in C^1(\R)$ a non-decreasing function such that, for given $\varepsilon>0$,
\[
\gamma_\varepsilon(t):=
\begin{cases}
1,\quad if\ 0<\varepsilon<t\\
0,\quad if\ t\le0, 
\end{cases}
\]
and consider, for $t_0\ge0$ given,
\[
f(x):=\gamma_\varepsilon\lp h(x)-t_0\rp,
\]
then, applying \eqref{eq5.6} with $f$ instead of $h$, and setting $S_{\alpha}:=\left\{x\in S:h(x)>\alpha\right\}$, we get,
\begin{align*}
\|V\|\lp S_{t_0+\varepsilon}\rp&=\int_{\left\{x\in M:h(x)-t_0>\varepsilon\right\}}h\dv\\
&=\int_{\left\{x\in S:f(x)\ge1\right\}}f\dv\\
&\le5^k\lp\frac1{\omega_k}\int_Sf\dv\rp^{\frac1k}\int_{S}\lp\gamma_\varepsilon'(h-t_0)\lmo\nabla^S_g h\rmo_g+f\lp\lmo H_g\rmo_g-c^*k\rp\rp\dv\\
&\le\frac{5^k}{\omega_k^{\frac1k}}\lp\|V\|(S_{t_0})\rp^{\frac1k}\int_{S}\lp\gamma_\varepsilon'(h-t_0)\lmo\nabla^S_g h\rmo_g+f\lp\lmo H_g\rmo_g-c^*k\rp\rp\dv,
\end{align*}
now, multiplying by $\lp t_0+\varepsilon\rp^{\frac1{k-1}}$ we have,
\tiny
\begin{align*}
\lp t_0+\varepsilon\rp^{\frac1{k-1}}\|V\|\lp S_{t_0+\varepsilon}\rp
&\le\frac{5^k}{\omega_k^{\frac1k}}\lp\lp t_0+\varepsilon\rp^{\frac k{k-1}}\|V\|(S_{t_0})\rp^{\frac1k}\int_{S}\lp\gamma_\varepsilon'(h-t_0)\lmo\nabla^S_g h\rmo_g+f\lp\lmo H\rmo_g-c^*k\rp\rp\dv\\
&=\frac{5^k}{\omega_k^{\frac1k}}\lp \int_{S_{t_0}}\lp t_0+\varepsilon\rp^{\frac k{k-1}}\dv\rp^{\frac1k}\int_{S}\lp\gamma_\varepsilon'(h-t_0)\lmo\nabla^S_g h\rmo_g+f\lp\lmo H\rmo_g-c^*k\rp\rp\dv\\
&\le\frac{5^k}{\omega_k^{\frac1k}}\lp \int_{S_{t_0}}\lp h+\varepsilon\rp^{\frac k{k-1}}\dv\rp^{\frac1k}\lp\int_{S}-\frac{\partial}{\partial t_0}\gamma_\varepsilon'\lp h-t_0\rp\lmo\nabla^S_g h\rmo_g\dv+\int_Sf\lp\lmo H\rmo_g-c^*k\rp\dv\rp\\
&\le\frac{5^k}{\omega_k^{\frac1k}}\lp \int_{S}\lp h+\varepsilon\rp^{\frac k{k-1}}\dv\rp^{\frac1k}\lp\int_{S}-\frac{\partial}{\partial t_0}\gamma_\varepsilon'\lp h-t_0\rp\lmo\nabla^S_g h\rmo_g\dv+\int_Sf\lp\lmo H\rmo_g-c^*k\rp\dv\rp.
\end{align*}
\normalsize
Integrating the above inequality on $t_0$ in the interval $]0,\infty[$, we have
\small
\begin{multline*}
\int_{0}^{\infty}\lp t_0+\varepsilon\rp^{\frac1{k-1}}\|V\|\lp S_{t_0+\varepsilon}\rp dt_0\le\\
\frac{5^k}{\omega_k^{\frac1k}}\lp \int_{S}\lp h+\varepsilon\rp^{\frac k{k-1}}\dv\rp^{\frac1k}
\int_0^{\infty}\lp\int_{S}-\frac{\partial}{\partial t_0}\gamma_\varepsilon'\lp h-t_0\rp\lmo\nabla^S_g h\rmo_g\dv+\int_Sf\lp\lmo H\rmo_g-c^*k\rp\dv\rp dt_0.
\end{multline*}
\normalsize
First, notice that the left hand side, by the Lemma \ref{riem4.3.5}, is equal to,
\[
\int_0^{\infty}\lp t_0+\varepsilon\rp^{\frac1{k-1}}\|V\|\lp S_{t_0+\varepsilon}\rp dt_0=\frac{k-1}k\int_{S_0}\lp h-\varepsilon\rp^{\frac k{k-1}}\dv=\frac{k-1}k\int_{\left\{x\in S:h(x)>\varepsilon \right\}}\lp h-\varepsilon\rp^{\frac k{k-1}}\dv.
\]
on the other hand, we can estimate the right hand side as follows, first by Fubini's Theorem
\begin{align*}
\int_0^{\infty}\lp\int_{S}-\frac{\partial}{\partial t_0}\gamma_\varepsilon'\lp h(x)-t_0\rp\lmo\nabla^S_g h\rmo_g\dv\rp dt_0&=\int_S\lmo\nabla^S_g h\rmo_g\lp-\frac d{dt}\int_0^{\infty}\gamma_\varepsilon\lp h(x)-t_0\rp dt_0\rp\dv\\
&=\int_S\lmo\nabla^S_g h\rmo_g\lp\gamma_\varepsilon(h(x))\rp\dv\\
&\le\int_S\lmo\nabla^S_g h(x)\rmo\dv.
\end{align*}
And, again by Fubini's Theorem,
\begin{align*}
\int_0^{\infty}\lp\int_Sf(x)\lp\lmo H_g\rmo_g-c^*k\rp\dv(x) \rp dt&=\int_S\lp\lmo H_g\rmo_g-c^*k\rp\lp\int_0^{\infty}\gamma_\varepsilon(h(x)-t_0)dt_0\rp\dv\\
&=\int_S\lp\lmo H_g\rmo_g-c^*k\rp\lp\int_{h(x)}^{-\infty}-\gamma_\varepsilon(w)dw\rp\dv\\
&\le\int_Sh\lp\lmo H_g\rmo_g-c^*k\rp\dv.
\end{align*}
Therefore
\[
\frac{k-1}k\int_{M_{\varepsilon}}\lp h-\varepsilon\rp^{\frac k{k-1}}\dv\le\frac{5^k}{\omega_k^{\frac1k}}\lp\int_M\lp h+\varepsilon\rp^{\frac k{k-1}}\dv\rp^{\frac1k}\int_M\lp\lmo\nabla^S_g h\rmo_g+h\lp\lmo H_g\rmo_g-c^*k\rp\rp\dv,
\]
then, letting $\varepsilon\to0$ we obtain
\[
\lp\int_Sh^{\frac k{k-1}}\dv\rp^{\frac{k-1}{k}}\le C\int_S\lp\lmo\nabla^S_g h\rmo_g+h\lp\lmo H_g\rmo_g-c^*k\rp\rp\dv,
\]
where
\[
C:=C(k):=\frac{k5^k}{(k-1)\omega_k^{\frac1k}}.
\]
\end{proof}

 \newpage
      \markboth{Bibliography}{Bibliography}
      \bibliographystyle{alpha}
      \bibliography{bibliografia}
      \addcontentsline{toc}{section}{\numberline{}Bibliography}

\end{document}